\documentclass[11pt]{amsart}
\usepackage{fullpage}
\usepackage{amssymb,amsmath,amsfonts,amsthm,mathrsfs}
\usepackage[all]{xy} \SelectTips{cm}{10}
\usepackage{color}

\numberwithin{equation}{section}

\renewcommand{\AA}{\mathbb A}

\newcommand{\CC}{\mathbb C}

\newcommand{\FF}{\mathbb F}
\newcommand{\GG}{\mathbb G}

\newcommand{\QQ}{\mathbb Q}
\newcommand{\RR}{\mathbb R}

\newcommand{\ZZ}{\mathbb Z}

\newcommand{\OO}{\mathcal O}

\newcommand{\calG}{\mathcal G}

\def\Aa{\mathcal A}

\newcommand{\F}{\mathscr F}

\newcommand{\Ss}{\mathscr S}
\newcommand{\calS}{\mathcal S}

\def\Zz{\mathcal Z}

\newcommand{\p}{\mathfrak p}
\renewcommand{\P}{\mathfrak P}
 
\newcommand{\Sym}{\mathfrak S}
\def\power{\mathcal{P}}

\newcommand{\norm}[1]{ \left|\!\left| #1 \right|\!\right|  }
\newcommand{\ang}[1]{ \langle #1 \rangle  }

\def\Spec{\operatorname{Spec}} 

 \def\Gal{\operatorname{Gal}}
\def\End{\operatorname{End}}

\def\Spec{\operatorname{Spec}}

\def\tr{\operatorname{tr}}

\def\Gal{\operatorname{Gal}}

\def \Li {\operatorname{Li}}  
\def \GL {\operatorname{GL}}  
\def \PGL {\operatorname{PGL}}
\def \SL {\operatorname{SL}}

\def \Sp {\operatorname{Sp}}
\def \GSp {\operatorname{GSp}}
\def\Aut{\operatorname{Aut}} 
\def\End{\operatorname{End}}
\def\Frob{\operatorname{Frob}}
\def\lcm{\operatorname{lcm}}
\renewcommand{\Re}{\operatorname{Re}}
\newcommand{\tors}{{\operatorname{tors}}}

 \def\Irr{\operatorname{Irr}}
\def\Prim{\operatorname{Prim}} 
\def\Fr{\operatorname{Frob}}
\def\tr{\operatorname{tr}}

\def\lcm{\operatorname{lcm}} \def\Li{\operatorname{Li}}

 \def \Aut {\operatorname{Aut}}
\def \mult{\operatorname{m}} 

\def\bbar#1{\setbox0=\hbox{$#1$}\dimen0=.2\ht0 \kern\dimen0 
\overline{\kern-\dimen0 #1}}
\newcommand{\Qbar}{{\overline{\mathbb Q}}} 
 
\newcommand{\kbar}{\bbar{k}}

\newcommand{\defi}[1]{\emph{#1}} 
\newtheorem{thm}{Theorem}[section]
\newtheorem{lemma}[thm]{Lemma}

\newtheorem{prop}[thm]{Proposition}

\theoremstyle{definition}
\newtheorem{defn}[thm]{Definition}

\newtheorem{conj}[thm]{Conjecture}

\theoremstyle{remark}
\newtheorem{remark}[thm]{Remark}
\newtheorem{example}[thm]{Example}

\newenvironment{romanenum}{\hfill \begin{enumerate} }{\end{enumerate}}

\definecolor{webbrown}{rgb}{.6,0,0}
\usepackage[
        colorlinks,
        linkcolor=webbrown,  filecolor=webcolor,  citecolor=webbrown, 
        backref,
        pdfauthor={David Zywina}, 
       pdftitle={The large sieve and Galois representations},
]{hyperref}
\usepackage[alphabetic,backrefs,lite]{amsrefs} 

\begin{document}
\title[]{The Large Sieve and Galois Representations}
\tableofcontents 
\author{David Zywina}
\address{Department of Mathematics, University of Pennsylvania, Philadelphia, PA 19104-6395, USA}
\email{zywina@math.upenn.edu}
\urladdr{http://www.math.upenn.edu/\~{}zywina}
\date{December 11, 2008}
  
\begin{abstract}
We describe a generalization of the large sieve to situations where the underlying groups are nonabelian, and give several applications to the arithmetic of abelian varieties.  In our applications, we sieve the set of primes via the system of representations arising from the Galois action on the torsion points of an abelian variety.  The resulting upper bounds require explicit character sum calculations, with stronger results holding if one assumes the Generalized Riemann Hypothesis. 
\end{abstract}

\maketitle

\section{Introduction} \label{S:Introduction}
\subsection{Overview} 
In this paper, we explain how a sieve theoretic method called the \emph{large sieve} can be suitably generalized to study various sequences of primes occurring in arithmetic geometry.

In \S\ref{SS:Koblitz introduction}--\ref{SS:Chavdarov intro}, we shall give some applications of our sieve (proofs can be found in \S\ref{S:Koblitz}--\ref{S:Chavdarov}).   These examples can be read independently of the rest of the paper and make no explicit reference to sieve theory or Galois representations.   Our choice of applications is not meant to be exhaustive but simply demonstrate a few basic problems that can be attacked with sieve theoretic methods.  

The large sieve is an important tool from analytic number theory (see \cite{Bombieri} and \cite{Montgomery} for background on the classical theory); our abstract form of the large sieve will be given in \S\ref{S:LS}.  The proof is similar to the classical version except one needs to be slightly careful since the underlying groups may not be abelian.   E.~Kowalski has independently come up with similar methods and has greatly generalized the large sieve (while adapting several of the ideas from this paper).  The reader is strongly encouraged to look at his recently published book \cite{Kowalski-LS}, which nicely complements the material presented here.

In \S\ref{S:Frobenius Sieve}, we will specialize to the case of sieving primes by conditions indexed by a collection of independent Galois representations.  The proof requires estimating various character sums with stronger results being obtained if one assumes the Generalized Riemann Hypothesis.   One of the merits of the large sieve is that these calculations need only be done once and we hope that Theorem~\ref{T:LS for Frobenius} will be of practical use for others.

Our applications will be proven in a common manner.   We first express the problem in terms of an independent system of strictly compatible Galois representations; these Galois representations give constraints on the set of primes that we are interested in.  The large sieve allows us to combine these constraints intelligently.  The reader interested in applying the sieve can skip directly to \S\ref{S:Frobenius Sieve}.

A quick remark is in order for analytic number theorists.  By \emph{large sieve}, we are referring to the sieve theoretic method of that name and not to the related inequalities.   The large sieve inequalities in this paper are all proved in a very naive manner.   Since the large sieve inequality deals with ``on average'' behaviour, one would hope to be able to prove stronger unconditional versions (due to the nonabelian nature of the our examples, it seems difficult to generalize the usual harmonic analysis arguments).   It would be interesting if someone could make any major improvement on our unconditional bounds.\\

We now introduce some notation that will hold throughout (further notation and conventions can be found in \S\ref{S:notation}).  For a number field $k$, denote its ring of integers by $\OO_k$.  Let $\Sigma_k$ be the set of non-zero prime ideals of $\OO_k$.  For each prime $\p\in \Sigma_k$, we have a residue field $\FF_\p=\OO_k/\p$ whose cardinality we denote by $N(\p)$.    Let $\Sigma_k(x)$  be the set of primes $\p$ in $\Sigma_k$ with $N(\p)\leq x$.

\subsection{Application: The Koblitz conjecture} \label{SS:Koblitz introduction}
Let $E$ be an elliptic curve without complex multiplication defined over a number field $k$.  Let $S_E$ be the set of places of $k$ for which $E$ has bad reduction.  For each $\p\in\Sigma_k-S_E$, denote the reduction of $E$ modulo $\p$ by $E_\p$. 

For all but finitely many $\p\in \Sigma_k-S_E$, reduction induces an injective homomorphism
$E(k)_{\tors} \hookrightarrow E_\p(\FF_\p)$ and in particular $|E(k)_{\tors}|$ divides $|E_\p(\FF_\p)|$.  Define the integer
\[
t_{E,k} = \lcm_{E'}|E'(k)_{\tors}|
\]
where the $E'$ vary over all elliptic curves  over $k$ that are $k$-isogenous to $E$.  The integer $|E_\p(\FF_\p)|$ is a $k$-isogeny invariant of $E$, so $t_{E,k}$ divides $|E_\p(\FF_\p)|$ for almost all $\p\in\Sigma_k-S_E$.  We are led to the following generalization of a conjecture of Koblitz (see \cite{Koblitz, Zywina-Koblitz}).
\begin{conj}
\label{C:Koblitz} 
Let $E$ be an elliptic curve over a number field $k$ without complex multiplication.  There is an explicit constant $C_{E,k}>0$ such that
\[
P_{E,k}(x):=|\{\p \in \Sigma_k(x)-S_E: |E_\p(\FF_\p)|/t_{E,k} \text{ is prime}\}|
 \sim C_{E,k} \frac{x}{(\log x)^2}
\]
as $x\to \infty$.
\end{conj}
\begin{remark}
There exists an elliptic curve $E'/k$ isogenous to $E$ over $k$ such that $|E'(k)_{\tors}|=t_{E,k}$.  Thus the conjecture can be restated in terms of counting the number of $\p$ such that the group $E'_\p(\FF_\p)/E'(k)_{\tors}$ has prime cardinality.  The motivation for the conjecture comes from elliptic curve cryptography where the discrete logarithm problem for $E(\FF_\p)$ is hardest when the cardinality is divisible by large primes.  There are no examples for which $\lim_{x\to \infty} P_{E,k}(x)=\infty$ has been proved.

Conjecture \ref{C:Koblitz} is given in \cite{Koblitz} under the assumptions that $k=\QQ$ and $t_{E,\QQ}=1$.  For heuristics and a description of the constant $C_{E,k}$ in Conjecture \ref{C:Koblitz}, see \cite{Zywina-Koblitz}.  The constant derived in \cite{Koblitz} is slightly different since it fails to take into account that the $\ell$-adic representations coming from the Galois action on the torsion points of $E$ need not be independent.   
The constant $C_{E,k}$ will be described explicitly in \S\ref{SS:Koblitz L} since it naturally occurs in the proof of Theorem~\ref{T:Koblitz bound}.

\end{remark}
Using the large sieve, we obtain the following upper bounds for $P_{E,k}(x)$.
\begin{thm} \label{T:Koblitz bound}
Let $E$ be an elliptic curve without complex multiplication defined over a number field $k$.
\begin{romanenum}
\item Then
\[
P_{E,k}(x) \leq (24 +o(1)) C_{E,k}\frac{x}{(\log x) (\log\log x)},
\]
where the $o(1)$ term depends on $E/k$.
\item Assuming the Generalized Riemann Hypothesis (GRH),
\[
P_{E,k}(x) \leq (22+o(1)) C_{E,k} \frac{x}{(\log x)^2},
\] 
where the $o(1)$ term depends on $E/k$.
\end{romanenum}
\end{thm}
\begin{remark}
Suppose $E/\QQ$ is a non-CM elliptic curve with $t_{E,\QQ}=1$.  In \cite{Cojocaru-Koblitz}, Cojocaru proves that $P_{E,\QQ}(x) \ll {x}/{(\log x)^2}$ assuming GRH\footnote{More precisely, Cojocaru needs only the $\theta$-quasi GRH  for some $1/2\leq \theta <1$; i.e., no Dedekind zeta function has a zero with real part greater that $\theta$.   If we assume only the $\theta$-quasi GRH, then our methods yield Theorem~\ref{T:Koblitz bound}(ii) with $22$ replaced by a larger constant depending on $\theta$.}.   The implicit constant depends on the conductor of $E$, but the exact dependency is not worked out.  Cojocaru's bound is proved using the Selberg sieve.  

Unconditionally, Cojocaru proves $P_{E,\QQ}(x) \ll x/((\log x) (\log\log\log x))$.  Though our unconditional bound is stronger, it is still not good enough to prove the analogue of Brun's theorem concerning the convergence of the sum of the reciprocal of twin primes.  More precisely, it is unknown whether the sum
$\sum_{p, |E_p(\FF_p)| \text{ prime}} p^{-1}$
is convergent.  Using our upper bound and partial summation, we are only able to show that the sum has very slow growth:
\[\sum_{p\leq x, |E_p(\FF_p)| \text{ prime}} \frac{1}{p} \ll \log\log\log x.\]
\end{remark}

\subsection{Application: Elliptic curves and thin sets}
\subsubsection{Thin sets}  We recall the notion of a thin set, for more details see \cite{SerreTopics}*{\S3} or \cite{SerreMordellWeil}*{\S 9}.  Let $n$ be a positive integer.
\begin{defn}
A set $\Omega \subseteq \QQ^n=\AA^n(\QQ)$ is \defi{thin} if there exists a  variety $X$ defined over $\QQ$ and a morphism $\pi\colon X \to \AA^n_\QQ$ with the following properties:
\begin{romanenum}
\item $\Omega \subseteq \pi(X(\QQ)),$
\item The fibre of $\pi$ over the generic point of $\AA^n_\QQ$ is finite and $\pi$ has no rational section  defined over $\QQ$.
\end{romanenum}
\end{defn}
\noindent There are two special types of thin sets:
\begin{description}
\item[Type $1$]  $\Omega$ is contained in a proper closed subvariety of $\AA^n_\QQ$.
\item[Type $2$]  $\Omega\subseteq \pi(X(\QQ))$ where $X$ is an irreducible variety over $\QQ$ of dimension $n$ and $\pi\colon X\to \AA^n_\QQ$ is a dominant morphism of degree $d\geq 2$.
\end{description}
Every thin subset of $\QQ^n$ is contained in a finite union of thin sets of Type $1$ and Type $2$.  

\subsubsection{Bounds}
Let $E$ be an elliptic curve defined over a number field $k$.  For each prime $\p\in\Sigma_k$, let $a_\p(E)$ be the corresponding \defi{trace of Frobenius}.  If $E$ has good reduction at $\p$, then $a_\p(E)=N(\p)+1 - |E_\p(\FF_\p)|$.

\begin{thm} \label{T:LS for elliptic curves}
Let $E_1,\ldots, E_n$ be elliptic curves without complex multiplication defined over a number field $k$ and assume that the $E_i$ are pairwise non-isogenous over $\kbar$.  Let $\Omega$ be a thin subset of $\ZZ^{n+1}$.  Then 
\begin{align*}
|\{\p \in \Sigma_k(x) :  (a_\p(E_1),\ldots,a_\p(E_n),N(\p))\in \Omega \}| &\ll \begin{cases}
     \displaystyle\frac{x(\log\log x)^{1+1/(9n+3)}}{(\log x)^{1+1/(18n+6)} } &  \\[1.0em]
     \displaystyle x^{1-1/(14n+8)} (\log x)^{2/(7n+4)} & \text{ assuming GRH.}
\end{cases}\\
\intertext{If $\Omega$ is a thin set of Type $1$, then}
|\{\p \in \Sigma_k(x) :  (a_\p(E_1),\ldots,a_\p(E_n),N(\p))\in \Omega \}| &\ll \begin{cases}
     \displaystyle \frac{x(\log\log x)^{2/(3n+1)}(\log\log\log x)^{1/(3n+1)}}{(\log x)^{1+1/(3n+1)}} &  \\[1.0em]
     \displaystyle \frac{x^{1-1/(6n+2)}}{(\log x)^{1 - 2/(3n+1)}}  \qquad\text{ assuming GRH.} &
     \end{cases}
\end{align*}
The implicit constants depend on the $E_i$, $k$, and $\Omega$.
\end{thm}

\begin{remark}
\begin{romanenum}
\item  
Absorbing the extra factors, our theorem gives bounds of the form $x/(\log x)^{1+\gamma}$ (and $x^{1-\delta}$ under GRH) with explicit values $\gamma,\delta >0$.

Note that while the implicit constant of Theorem~\ref{T:LS for elliptic curves} depends on the thin set $\Omega$, the function of $x$ does not.  It is natural to ask what the optimal function of $x$ could be?  The example of \S\ref{SS:toy example} suggests that the general bound is at best $x^{3/4}/(\log x)$.

\item
Theorem \ref{T:LS for elliptic curves} was inspired by published remarks of Serre.  Remark 2 of \cite{SerreTopics}*{\S3.6} states (but does not prove) the unconditional case of the theorem for a single elliptic curve defined over $\QQ$ (and does not describe the exponent).  A similar remark for the Lang-Trotter conjecture of Example~\ref{R:LT2} is given in \cite{SerreCheb}*{\S8.2 Remark 4}.

\item
Theorem~\ref{T:LS for elliptic curves} in the case of a thin set of Type $1$ is not proven using the large sieve.  In this case, instead of sieving by many primes it is better to sieve by a single well chosen prime.
\end{romanenum}
\end{remark}

Let us consider a few special cases of Theorem~\ref{T:LS for elliptic curves}.
\begin{example} \label{R:LT2}
Let $E/\QQ$ be a non-CM elliptic curve and $K$ an imaginary quadratic extension of $\QQ$.  For each prime $p$ of good reduction of $E$, let $\pi_{p}$ be the Frobenius endomorphism of $E_p$ (it is a root of $t^2-a_p(E)t+p$). Define
\[\Pi_{E,K}(x) = |\{ p \leq x: E \text{ has good reduction at $p$, }  \QQ(\pi_p)\cong K\}|.\]
The \emph{Lang-Trotter conjecture} \cite{Lang-Trotter} predicts that there is a constant $C>0$, depending on $E$ and $K$, such that
\[ 
\Pi_{E,K}(x) \sim C \frac{x^{1/2}}{\log x} 
\]
as $x\to\infty$ (another conjecture of Lang and Trotter will be described in \S\ref{S:Lang-Trotter}).  Let $D_K$ be the discriminant of $K$ and define 
\[
\Omega = \{ (a,b)\in \ZZ^2 : a^2-4b = D_Kc^2 \text{ for some }c\in\QQ^\times\}.
\]
Define $X=\Spec(\QQ[x,y,z]/(x^2-4y - D_Kz^2))$ and the morphism
\begin{align*}
\pi\colon X &\to \Spec \QQ[x,y] = \AA^2_\QQ, \quad (x,y,z) \mapsto (x,y).
\end{align*}
The set $\Omega\subseteq \QQ^2$ is thin of Type 2 since $\Omega\subseteq \pi(X(\QQ))$, $X$ is irreducible of dimension $2$, and $\pi$ is a dominant map of degree $2$.  We have $\Pi_{E,K}(x) = |\{ p\leq x:  (a_p(E),p) \in\Omega \}|+O(1),$ and by Theorem~\ref{T:LS for elliptic curves}
\[ 
\Pi_{E,K}(x) \ll \begin{cases}
     x(\log\log x)^{13/12}/(\log x)^{25/24}  &  \\
     x^{21/22} (\log x)^{2/11} & \text{ assuming GRH. }
\end{cases}
\]
Better bounds for this particular example can be found in \cite{CojocaruDavid-LT2} or  \cite{Zywina-LT2}.
\end{example}

\begin{example}
Let $E$ and $E'$ be non-CM elliptic curves over a number field $k$ which are non-isogenous over $\kbar$.  Define the set $\Omega=\{(a,b,c)\in\ZZ^3: a=b\}$ which is thin of Type 1.  Theorem~\ref{T:LS for elliptic curves} becomes
\[
|\{\p\in\Sigma_k(x):  a_\p(E)=a_\p(E') \}| 
\ll  \begin{cases}
     \displaystyle \frac{x(\log\log x \cdot \log\log\log x)^{1/4}}{(\log x)^{9/8}} &  \\[1.0em]
     \displaystyle \frac{x^{13/14}}{(\log x)^{5/7}} & \text{ assuming GRH.} 
     \end{cases}
\]
This gives an explicit version of a theorem of Faltings which shows that the values $a_\p(E)$ determine the isogeny class of $E$. 
(That such a theorem can be deduced is not surprising given that work of Faltings is needed in the proof to describe the image of the corresponding Galois representations.)
\end{example}

\subsection{Application: Abelian varieties and Galois groups of characteristic polynomials} 
\label{SS:Chavdarov intro}
\begin{defn}
Fix a polynomial $P(T) \in \QQ[T]$.  The \defi{Galois group} of $P$ is defined to be $\Gal(P):=\Gal(L/\QQ)$ where $L$ is the splitting field of $P(T)$ in a fixed algebraic closure $\Qbar$ of $\QQ$.
\end{defn}
\subsubsection{Abelian varieties over finite fields}
\label{SSS:Abelian varieties over finite fields}
Let $A$ be an abelian variety of dimension $g$ defined over a finite field $\FF$ with $q$ elements.  Let $\pi_A$ be the $q$-power Frobenius endomorphism of $A$.  
There is a unique polynomial $P_A(T)\in\ZZ[T]$ of degree $2g$ such that the isogeny $r-\pi_A$ of $A$ has degree $P_A(r)$ for $r\in \ZZ$.  The polynomial $P_A(T)$  satisfies the functional equation,
\[
P_A(q/T)/(q/T)^g = P_A(T)/T^g.
\]
From the functional equation we find that if $\pi$ is a root of $P_A(T)$, then so is $q/\pi$.  Let $\pi_1,\dots,\pi_{2r}$ be the distinct non-rational roots of $P_A(T)$ in $\Qbar$; we may assume that they are numbered so that $\pi_{2i-1}\pi_{2i}=q$ or $\{\pi_{2i-1},\pi_{2i}\}=\{\pm\sqrt{q}\}$ for $1\leq i\leq r$.  The Galois group $\Gal(P_A(T))$ acts on the roots of $P_A(T)$ and induces an action on the $r$ pairs $\{\pi_{1},\pi_2\},\dots, \{\pi_{2r-1},\pi_{2r}\}$.  

\begin{defn}
Let $W_{2r}$ be the group of permutations of $\{1,\dots,2r\}$ which induce a permutation of the set $\big\{\{1,2\},\{3,4\},\dots,\{2r-1,2r\} \big\}$.
\end{defn}
The numbering of the $\pi_i$'s gives an injective homomorphism $\Gal(P_A(T)) \hookrightarrow W_{2r}$.  In particular, we find that $\Gal(P_A(T))$ is isomorphic to a subgroup of $W_{2g}$.  Thus the largest possible Galois group for the polynomial $P_A(T)$ is $W_{2g}$.  The group $W_{2g}$ has order $2^g g!$ and is isomorphic to the Weyl group of $\Sp(2g)$.

\subsubsection{Explicit Chavdarov}
Fix an abelian variety $A$ defined over a number field $k$ and let $S_{A}\subseteq \Sigma_k$ be the set of prime ideals for which $A$ has bad reduction.  For each $\p\in\Sigma_k-S_A$, let $A_\p$ be the abelian variety over $\FF_\p$ obtained by reduction modulo $\p$.  For an integer $n\geq 1$, let $\FF_\p^{(n)}$ be the degree $n$ field extension of $\FF_\p$ and let $A_\p \times \FF_\p^{(n)}$ be the base extension of $A_\p$ by $\FF_\p^{(n)}.$

We define $\Pi_A$ to be the set of $\p\in\Sigma_k-S_A$ such that
\[
\Gal\!\big(P_{A_\p\times {\FF_\p^{(n)}}}(T)\big)\not\cong W_{2g}
\]  
for some $n\geq 1$.  The following result of Chavdarov \cite{Chavdarov}*{Corollary~6.9} shows that $\Pi_A$ has natural density $0$ for certain abelian varieties.  Define $\Pi_A(x)=\Pi_A\cap \Sigma_k(x)$. 
\begin{thm}[Chavdarov]\label{T:Chavdarov} 
Let $A$ be an abelian variety of dimension $g$ defined over a number field $k$.  Suppose that $g$ is either $2$, $6$ or odd,  and $\End_{\kbar}(A)=\ZZ$.  Then 
\[
\lim_{x\to \infty} {|\Pi_A(x)|}/{|\Sigma_k(x)|} = 0.
\]
In other words, the primes $\p\in\Sigma_k-S_A$ for which $\Gal\!\big(P_{A_\p\times {\FF_\p^{(n)}}}(T)\big)\not\cong W_{2g}$ for all $n\geq 1$, have natural density $1$.
\end{thm} 

The following theorem, which will be proven with the large sieve, gives an explicit version of Chavdarov's theorem.
\begin{thm} \label{T:Effective Chavdarov}
Let $A$ be an abelian variety of dimension $g$ defined over a number field $k$.  Suppose that $g$ is either $2$, $6$ or odd,  and $\End_{\kbar}(A)=\ZZ$.  Then
\begin{align*}
|\Pi_A(x)| &\ll 
\frac{x(\log\log x)^{1+1/(6g^2+3g+3)}}{(\log x)^{1+1/(12g^2+6g+6)} },
\intertext{and assuming the Generalized Riemann Hypothesis}
|\Pi_A(x)| &\ll 
x^{1-1/(8g^2+6g+8)} (\log x)^{2/(4g^2+3g+4)}.
\end{align*}
The implicit constants depend on $A/k$.
\end{thm}

\begin{remark}
\begin{romanenum}
\item  
Theorem~\ref{T:Effective Chavdarov} can be used to bound the number of $\p$ for which $A_\p$ is not geometrically simple.  Fix a prime $\p\in\Sigma_k -(S\cup \Pi_A)$.  For each $n\geq 1$, the polynomial $P_{A_\p\times \FF_{\p}^{(n)}}(T)$ is irreducible since $\Gal(P_{A_\p\times \FF_{\p}^{(n)}}(T))\cong W_{2g}$ acts transitively on its $2g$ roots.  We also deduce that $\QQ(\pi_{A_\p}^n) = \QQ(\pi_{A_\p})$ for all $n\geq 1$.

By \cite{MilneWaterhouse}*{Theorem~8}, for each $n\geq 1$ we have
\[
\End_{\FF_{\p}^{(n)}}(A_\p)\otimes_\ZZ\QQ = \QQ(\pi_{A_\p}^n)
\] 
and hence $\End_{\FF_{\p}^{(n)}}(A_\p)\otimes_\ZZ\QQ= \QQ(\pi_{A_\p})$.  We deduce that $\End_{\overline{\FF}_\p}(A_\p) \otimes_\ZZ \QQ= \QQ(\pi_{A_\p})$ and thus $A_\p/\FF_\p$ is geometrically simple.  We then have an inequality
\[
|\{\p\in\Sigma_k(x)-S_A: A_\p \text{ \emph{not} geometrically simple}\}| \leq |\Pi_A(x)|
\]
and Theorem~\ref{T:Effective Chavdarov} gives an explicit upper bound.
\item 
The dimension assumptions on the abelian varieties are needed only to invoke a theorem of Serre which says that $\Gal(k(A[\ell])/k) \cong \GSp_{2g}(\ZZ/\ell\ZZ)$ for all sufficiently large primes $\ell$.   This condition will hold for a ``random'' abelian variety $A/k$ of any dimension $g$.  See the recent paper of Hall \cite{Hall} which gives a nice sufficient condition to have $\Gal(k(A[\ell])/k) \cong \GSp_{2g}(\ZZ/\ell\ZZ)$ for almost all $\ell$.

\item 
The majority of \cite{Chavdarov} deals with the function field setting in which Galois representations can be identified with representations of \'etale fundamental groups.  The large sieve method is applicable in this context as well and this has already been studied in a paper of Kowalski \cite{Kowalski-Monodromy} (see Chapter~8 of \cite{Kowalski-LS}).  The large sieve presented in this paper would need to be altered slightly to deal properly with both the arithmetic and geometric fundamental groups.  In the function field setting one can prove strong unconditional bounds since one may use the full force of the Weil conjectures.
\end{romanenum}
\end{remark}

\subsection{An explanatory example} \label{SS:toy example}
We shall now illustrate the basic concepts underlying the paper with a simple (but nontrivial!) example.  The reader may safely skip ahead.   \\

Fix an elliptic curve $E$ defined over $\QQ$ and assume that $E$ does not have complex multiplication.  Let $S_E$ be the set of places of $k$ for which $E$ has bad reduction.  For each prime $p\notin S_E$, let $a_p(E)$ be the integer such that $|E_p(\FF_p)|=p-a_p(E)+1$
where $E_p/\FF_p$ is the reduction of $E$ at $p$.  In this example, we will study the set
\[
\Aa:= \{ p \notin S_E :  a_p(E) \text{ is a square}\}.
\]
The set $\Aa$ is infinite (Elkies has shown that there are infinitely many $p$ with $a_p(E)=0$ \cite{Elkies87}).  For each real number $x$, let $\Aa(x)$ be the set of $p\in\Aa$ with $p\leq x$.  We will see that $\Aa$ has natural density zero; what is more interesting is to find explicit bounds for $|\Aa(x)|$.  

Crude heuristics suggest that there is a constant $C>0$, depending on $E$, such that
\[
|\Aa(x)| \sim C \frac{x^{3/4}}{\log x}
\]
as $x\to\infty$.  Proving anything like this is exceedingly difficult;  we will focus on finding upper bounds for $|\Aa(x)|$.\\

The basic idea is to study the integers $a_p(E)$ modulo several small primes $\ell$ and then combine this local information to find an explicit upper bound for $|\Aa(x)|$.  To understand the distribution of the $a_p(E)$ modulo $\ell$, it is advantageous to express everything in terms of Galois representations.  For each prime $\ell$, let $E[\ell]$ be the group of $\ell$-torsion points in $E(\Qbar)$.  The absolute Galois group of $\QQ$ naturally acts on $E[\ell]$ giving a representation
\[
\rho_{E,\ell}\colon \Gal(\Qbar/\QQ) \to \Aut(E[\ell]) \cong \GL_2(\ZZ/\ell\ZZ).
\]
From Serre (Theorem~\ref{T:Abelian variety with full GSp} with $g=1$), we know that there is a positive integer $B$ such that 
\begin{equation} \label{E:SerreInv intro}
\Bigl(\prod_{\ell\nmid B}\rho_{E,\ell}\Bigr)(\Gal(\Qbar/\QQ)) = \prod_{\ell\nmid B} \GL_2(\ZZ/\ell\ZZ).
\end{equation}
Fix a prime $\ell\nmid B$ and take any prime $p\notin S_E\cup\{\ell\}$.  The Galois representation $\rho_{E,\ell}$ is unramified at $p$, so we obtain a well-defined conjugacy class $\rho_{E,\ell}(\Frob_p)$ of $\GL_2(\ZZ/\ell\ZZ)$.   The connection with the integer $a_p(E)$ is the congruence
\[
\tr(\rho_\ell(\Frob_p)) \equiv a_p(E) \bmod{\ell}.
\]
Now take any prime $p\in \Aa(x)$ with $p\neq \ell$.  Since $a_p(E)$ is a square, the trace of the Frobenius conjugacy class $\rho_{E,\ell}(\Frob_p)$ is a square in $\ZZ/\ell\ZZ$.  Therefore,
\[
\rho_{E,\ell}(\Frob_p)\subseteq C_\ell:= \{ A\in \GL_2(\ZZ/\ell\ZZ) : \tr(A) \text{ is a square in }\ZZ/\ell\ZZ\}. 
\]  
One readily checks that 
\begin{equation} \label{E:toy probability}
\frac{|C_\ell|}{|\GL_2(\ZZ/\ell\ZZ)|} = \frac{1}{2} + O\left(\frac{1}{\ell} \right).\\
\end{equation}

Let us now describe our example in a sieve theoretic fashion. 
Let $Q=Q(x)$ be a positive function such that $Q(x)\ll \sqrt{x}$; we will make a specific choice later.
Let $\Lambda(Q)$ be the set of primes $\ell\nmid B$ with $\ell\leq Q$.

We now sieve the set $\Sigma_\QQ(x)-S_E$ by the primes $\ell\in\Lambda(Q)$.  More precisely, for each $\ell \in \Lambda(Q)$ we remove those primes $p$ for which $\tr(\rho_\ell(\Frob_p)) \in \ZZ/\ell\ZZ$ is a \emph{non-square}.   We are then left with the set
\begin{align*}
\Ss(x)
&= \{p \in \Sigma_\QQ(x)-S_E : p= \ell \; \text{ or } \; \rho_\ell(\Frob_p))\subseteq C_\ell \text{ for all } \ell\in \Lambda(Q) \}.
\end{align*}
The set $\Ss(x)$ contains $\Aa(x)$, so it suffices to consider upper bounds for $|\Ss(x)|$.

Intuitively, the Chebotarev density theorem and (\ref{E:toy probability}) tell us that sieving $\Sigma_\QQ(x)-S_E$ by a prime $\ell\in \Lambda(Q)$ will remove roughly half the elements, while (\ref{E:SerreInv intro}) shows that our sieving conditions (indexed by the primes $\ell\in\Lambda(Q)$) are  independent of each other.

Let $Q$ be a constant function.  The Chebotarev density theorem gives us
\[
\limsup_{x\to\infty} \frac{|\Ss(x)|}{x/\log x} \leq 
\prod_{\ell\in\Lambda(Q)} \frac{|C_\ell|}{|\GL_2(\ZZ/\ell\ZZ)|}
\leq \prod_{\ell\in\Lambda(Q)} \left(\frac{1}{2} + O\left(\frac{1}{\ell}\right)\right).
\]
Since this holds for every constant $Q$, we find that the set $\Aa$ has natural density $0$; i.e., \[\lim_{x\to\infty} \frac{|\Aa(x)|}{x/\log x} = 0.\] 
In the same manner, we can apply effective versions of the Chebotarev density theorem (as in Appendix~\ref{A:Character sums}) to obtain explicit upper bounds for $|\Aa(x)|$.  However, the resulting bounds will be weaker that those coming from the sieve theoretic methods discussed in this paper.  In particular, assuming the Generalized Riemann Hypothesis (GRH), they will not be strong enough to prove that there is a number $\delta>0$ such that $|\Aa(x)|\ll x^{1-\delta}$.

This direct approach requires equidistribution of the conjugacy classes $\{(\prod_{\ell\in\Lambda(Q)}\rho_\ell)(\Frob_p)\}_p$ in the conjugacy classes of $\prod_{\ell\in\Lambda(Q)} \GL_2(\ZZ/\ell\ZZ)$ (with respect to the measure induces by Haar measure).   The prime number theorem shows that the order of this group grows quickly as a function of $Q$,
\begin{equation}\label{E:naive equidistribution}
\Bigl|{\prod}_{\ell\in\Lambda(Q)} \GL_2(\ZZ/\ell\ZZ)\Bigr| = e^{4Q+o(Q)}.
\end{equation}
So in order for the Frobenius elements $\{({\prod}_{\ell\in\Lambda(Q)}\rho_\ell)(\Frob_p)\}_{p\leq x}$ to be well equidistributed, we need the function $Q(x)$ to grow quite slowly as a function of $x$. \\

Let us now discuss what the large sieve method will give. (We will be applying the large sieve as in Theorem~\ref{T:LS for Frobenius}.  The details are similar to those given in \S\ref{S:Koblitz} for our application to a conjecture of Koblitz.)  The advantage over the direct approach just given is that it allows one to limit the size of the groups considered. Let $\Zz(Q)$ be the set of $D\subseteq\Lambda(Q)$ such that $\prod_{\ell\in D} \ell \leq Q$ and define
\[
L(Q) = \sum_{D\in \Zz(Q)} \prod_{\ell\in D} \frac{1-{|C_\ell|}/{|\GL_2(\ZZ/\ell\ZZ)|} }{{|C_\ell|}/{|\GL_2(\ZZ/\ell\ZZ)|}}.
\]
Using (\ref{E:toy probability}) one shows that $L(Q) \gg Q$.
Assuming GRH, our large sieve will give the bound 
\begin{align*}
 |\Ss(x)| &\ll \frac{x/\log x +  Q^{11} x^{1/2}\log x}{L(Q)}
\ll \frac{x/\log x +  Q^{11} x^{1/2}\log x}{Q}.
\end{align*}
Setting $Q(x)= x^{1/22}/(\log x)^{2/11}$ (this choice makes the two terms in the numerator of our bound equal), we have
\[
|\Aa(x)| \leq |\Ss(x)| \ll \frac{x^{21/22}}{(\log x)^{9/11}}.
\]
Unconditionally, with $Q(x) \approx (\log x/ (\log\log x)^2)^{1/24}$, our large sieve will give
\begin{align*}
|\Aa(x)| & \ll \frac{x/\log x}{L(Q)} \ll \frac{x(\log\log x)^{1/12}}{(\log x)^{25/24}}.
\end{align*}
An examination of the proof of the large sieve shows that in our example, we use equidistribution results only for the groups $\prod_{\ell\in D\cup D'}\GL_2(\ZZ/\ell\ZZ)$ with $D,D' \in \Zz(Q)$.  For $D,D'\in \Zz(Q)$,
\[
\Bigl|{\prod}_{\ell\in D\cup D'} \GL_2(\ZZ/\ell\ZZ)\Bigr| \leq {\prod}_{\ell\in D\cup D'} \ell^4 \leq Q^8.
\]
Thus to give a bound for $|\Aa(x)|$ using the primes $\ell \in \Lambda(Q)$, the large sieve uses equidistribution for groups of size at most $Q^8$ (this should be contrasted with (\ref{E:naive equidistribution})).

\subsection{Notation}   \label{S:notation}
For each field $k$, let $\kbar$ be an algebraic closure of $k$ and let $\calG_k := \Gal(\kbar/k)$ be the absolute Galois group of $k$.  


Let $L$ be a Galois extension of a number field $k$.  For each $\p\in\Sigma_k$ that is unramified in $L$, let $\Fr_\p$ be the Frobenius conjugacy class of $\p$ in $\Gal(L/k)$ (though the notation does not indicate it, the extension $L$ will always be clear from context).

Suppose that $f$ and $g$ are complex valued functions of a real variable $x$.  By $f\ll g$ (or $g \gg f$), we shall mean that there are positive constants $C_1$ and $C_2$ such that for all $x\geq C_1$, $|f(x)| \leq C_2 |g(x)|$.  The dependence of the implied constants will not always be given but will be made precise in the statement of each of the main theorems.  We shall use $O(f)$ to denote an unspecified function $g$ with $g \ll f$.  We shall write $f=o(g)$ if $g$ is nonzero for sufficiently large $x$ and $f(x)/g(x) \to 0$ as $x\to \infty$.

For $x \geq 2$, define the logarithmic integral $\Li x = \int_2^x (\log t)^{-1} dt$.  The function $\Li x$ is useful when counting primes, and $\Li x = (1+o(1)) x/\log x$ as $x\to \infty$.

For a finite group $G$, the set of conjugacy classes of $G$ will be denoted by $G^\sharp$.  We denote by $\Irr(G)$ the set of characters of $G$ which come from irreducible linear representations of $G$ over $\CC$.  Let $(\;,\;)$ be the inner product on the the space of complex-valued class functions of $G$ for which $\Irr(G)$ is an orthonormal basis.   For the basic notions of representation theory see \cite{Serre-rep}.

Finally, $\ell$ and $p$ will denote rational primes.

\subsection*{Acknowledgments}

Many thanks to Bjorn Poonen for his encouragement and assistance.   Thanks also to  Jeff Achter, Alina Cojocaru and Emmanuel Kowalski for their helpful suggestions.   This research was supported by an NSERC postgraduate scholarship.

\section{A general large sieve} \label{S:LS}
\subsection{Setup and statement} \label{S:LS setup}
Let $X$ and $\Lambda$ be finite sets.  We will sieve subsets of $X$ via conditions indexed by $\Lambda$.  

For each $\lambda\in \Lambda$, fix a finite group $G_\lambda$ and a map $\rho_\lambda\colon X \to G_\lambda^\sharp$.  Let $\mu_\lambda$ be the probability measure on $G_\lambda^\sharp$ induced by the counting (Haar) measure on $G_\lambda$.  More concretely, we have $\mu_\lambda(U) = |G_\lambda|^{-1} \sum_{C\in U} |C|$ for every subset $U\subseteq G_\lambda^\sharp$.

Consider a set $\Ss\subseteq X$.  The goal of our sieve is to find an upper bound for $|\Ss|$ in terms of the values $\mu_\lambda(\rho_\lambda(\Ss))$.\\

We now introduce some more notation.  For each $D\subseteq \Lambda$, define the group $G_D = \prod_{\lambda\in D} G_\lambda$.  For a subset $E\subseteq D$, composition with the projection $G_D\to G_E$ induces an injective map $\Irr(G_E) \hookrightarrow \Irr(G_D)$.  We say that a character $\chi \in \Irr(G_D)$ is \defi{imprimitive} if it comes from a character in $\Irr(G_E)$ for some {proper} subset $E$ of $D$; otherwise we say that $\chi$ is \defi{primitive}.  Let $\Prim(G_D)$ denote the set of primitive characters in $\Irr(G_D)$.  Define $\rho_D:= (\prod_{\lambda\in D} \rho_\lambda) \colon X \to G_D^\sharp$.  Finally, let $\mathcal{P}(\Lambda)$ be the set of all subsets of $\Lambda$.

\begin{thm}[Large sieve] \label{T:Large Sieve} Fix notation as above, and let $\Zz$ be a subset of $\power(\Lambda)$.  Let $\Delta(X,\rho,\Zz)\geq 0$ be the least real number for which the inequality
 \begin{equation} \label{E:LSI of theorem}
 \sum_{D\in\Zz} \sum_{\chi\in\Prim(G_D)}\left| \sum_{v\in X} a_v\chi(\rho_D(v))\right|^2 \leq \Delta(X,\rho,\Zz) \sum_{v\in X} |a_v|^2 \end{equation}
holds for every sequence $(a_v)_{v\in X}$ of complex numbers.  

Let $\Ss$ be a subset of $X$.  For each $\lambda\in \Lambda$, fix a real number $0<\delta_\lambda \leq 1$ such that
\begin{equation} \label{E:delta size}
\mu_\lambda(\rho_\lambda(\Ss)) \leq \delta_\lambda.
\end{equation}
Define
\[
L(\Zz) = \sum_{D\in \Zz} \prod_{\lambda\in D} \frac{1-\delta_\lambda}{\delta_\lambda}. 
\]
Then 
\[ 
L(\Zz) |\Ss| \leq \Delta(X,\rho,\Zz).
\]
\end{thm}

\begin{remark}
\begin{romanenum}
\item  In classical versions of the large sieve, the groups $G_\lambda$ are abelian and the set $\Lambda$ usually consists of prime 
numbers.  Theorem \ref{T:Large Sieve} shows that the basic sieve theoretic principle underlying the large sieve can be fruitfully 
generalized.
\item Note that $\Delta(X,\rho,\Zz)$ does not depend on the set $\Ss$ and the dependency of $L(\Zz)$ on $\Ss$ is only in terms 
of the $\delta_\lambda$. 
\item  An inequality of the form (\ref{E:LSI of theorem}) is called a \defi{large sieve inequality}.  The study of such inequalities 
has been very important for analytic number theory (cf.~\cite{IwaniecKowalski}*{\S7}).   Quite often in the literature, the term 
``large sieve'' refers to the study of large sieve type inequalities, even when no actual sieving is involved!  

Proposition \ref{P:LSI} below gives a very basic upper bound on $\Delta(X,\rho,\Zz)$.  In the abelian case one can often prove stronger large sieve equalities using harmonic analysis.  

\item  Theorem~\ref{T:Large Sieve} can be further generalized but this version will more than suffice for our applications.  For example, in the proof of Theorem~\ref{T:Large Sieve} we do not explicitly need the group structure of the groups $G_\lambda$.  The main fact we need is that the characters $\Irr(G_\lambda)$ form an orthonormal basis for the Hilbert space of class functions on $G_\lambda$.  We refer to the Kowalski's book \cite{Kowalski-LS} for further studies in this direction.   
\end{romanenum}
\end{remark}

\begin{prop} \label{P:LSI}
With notation as in Theorem \ref{T:Large Sieve},  \[
 \Delta(X,\rho,\Zz) \leq \max_{\substack{D'\in \Zz \\ \chi'\in \Prim(G_{D'})} }  \sum_{D\in\Zz} \sum_{\chi\in\Prim(G_D)}  \left| \sum_{v\in X} \chi(\rho_{D}(v)) \overline{\chi'(\rho_{D'}(v))} \right|.
 \] 
\end{prop}
A proof of the proposition can be found in \S\ref{SS:Proof of LSI}.

\begin{remark}
Thinking of the elements $\{\rho_{D\cup D'}(v)\}_{v\in X}$ as being equidistributed in $G_{D\cup D'}^\sharp$ with respect to the measure $\prod_{\lambda\in D\cup D'}\mu_\lambda$, we would expect the following expression to be small:
\[
\sum_{v\in X} \chi(\rho_D(v))\overline{\chi'(\rho_{D'}(v))} - \begin{cases}
      |X| & \text{if $\chi=\chi'$}, \\
       0 & \text{otherwise}.
\end{cases}
\]
The large sieve and Proposition~\ref{P:LSI} then gives an explicit bound of the form 
\[
|\Ss| \leq \frac{|X| + E(\Zz) }{L(\Zz)}
\]
where one can think of $E(\Zz)$ as being an ``error term''.

Taking $\Zz=\power(\Lambda)$, we find that $L(\Zz) = \prod_{\lambda\in\Lambda}(1 + (1-\delta_\lambda)/\delta_\lambda) =(\prod_{\lambda\in \Lambda} \delta_\lambda )^{-1}.$
Thus the ``main term'' of our bound is $|X|/L(\power(\Lambda))=( \prod_{\lambda\in\Lambda} \delta_\lambda ) |X|$; i.e.,  what one would naively expect after sieving $X$ by independent conditions, indexed by $\lambda\in \Lambda$, each with probability $\delta_\lambda$.  Unfortunately, taking $\Zz=\power(\Lambda)$ in most applications will not be useful since the ``error term'' will be too large.  The set $\Zz$  is called the \defi{sieve support} and should be chosen to optimize or simplify the bounds in a given application.  
\end{remark}

\subsection{The classical large sieve}
As a simple example, let us show how the large sieve of Theorem \ref{T:Large Sieve} relates to the familiar case of sieving integers.  Fix a natural number $N$ and real numbers $M$ and $Q\geq 2$.  Define the sets
\[
X=\{n\in\ZZ: M<n\leq M+N \} \text{\quad and \quad} \Lambda=\{\ell: \ell \text{ prime and }\ell\leq Q\}.
\]
For each $\ell\in\Lambda$, let $G_\ell$ be the group $\ZZ/\ell\ZZ$ and let $\rho_\ell \colon X \to G_\ell^\sharp  =\ZZ/\ell\ZZ$ be 
reduction modulo $\ell$.  The set $\power(\Lambda)$ can be identified with the squarefree natural numbers whose prime factors have size at most $Q$ (identify a 
squarefree natural number with the set of its prime divisors).  For each $d\in\power(\Lambda)$, $G_d = \prod_{\ell|d} G_\ell = 
\ZZ/d\ZZ$ and $\rho_d$ is simply reduction modulo $d$.  The irreducible characters of $G_d=\ZZ/d\ZZ$ are those of form $x\mapsto 
e^{2\pi i \cdot {ax}/{d}}$ for $a\in \ZZ/d\ZZ$.  The set $\Prim(G_d)$ consists of those characters with $a\in 
(\ZZ/d\ZZ)^\times$.

The classical choice for $\Zz$ is the set $ \{ d \in \power(\Lambda): d\leq Q \}$; the squarefree natural numbers less than 
or equal to $Q$.  The following lemma shows that in the above setting, $\Delta(X,\rho,\Zz) \leq N+Q^2$; it is a consequence of \cite{Bombieri}*{Th\'eor\`eme 4}.

\begin{lemma} \label{L:classical LSI}
For any sequence of complex numbers $(a_n)_{n\in X}$,
\[
\sum_{d\leq Q} \sum_{a\in (\ZZ/d\ZZ)^\times}\Bigl| \sum_{n\in X} a_n e^{2\pi i \cdot a n/d} \Bigr| \leq (N + Q^2) \sum_{n\in X} 
|a_n|^2.
\]
\end{lemma}
In the present case, Theorem \ref{T:Large Sieve} specializes to the following familiar version of the large sieve. 
\begin{thm} \label{T:classical LS} Let $\Ss$ be a set of integers contained in an interval of length $N\geq 1$.  Let $Q\geq 2$ be a real number.  For each prime $\ell\leq Q$, fix a number $0<\delta_\ell\leq 1$ such that 
${ |\{n \bmod{\ell}: n \in \Ss\}| } \leq \delta_\ell \ell$.
Then 
\[
|\Ss| \leq (N+Q^2) \Big(\sum_{d\leq Q \text{ squarefree}} \prod_{\ell|d}  \frac{1-\delta_\ell}{\delta_\ell} \Big)^{-1}.
\]
\end{thm}

\subsection{Proof of Theorem \ref{T:Large Sieve}} 
\label{S:LS proof}
\begin{lemma}  \label{L:algebra lemma} 
For any $D\subseteq\Lambda$, we have
\[
\Bigl( \prod_{\lambda\in D} \frac{1-\delta_\lambda}{\delta_\lambda} \Bigr) \Bigl| \sum_{v\in X} a_v\Bigr|^2 \leq \sum_{\chi\in \Prim(G_D)} \Bigl| \sum_{v\in X} a_v \chi(\rho_D(v)) \Bigr|^2
\]
where $(a_v)_{v\in X}$ is any sequence of complex numbers such that $a_v=0$ for all $v \in X-\Ss$.
\end{lemma}
\begin{proof}
We proceed by induction on the cardinality of the set $D$. \\ 
\noindent $\bullet$ If $|D|=0$, then $D=\emptyset$ and the lemma is trivial.  Note that $\Prim(G_\emptyset)=\{1\}$.\\
\noindent $\bullet$ If $|D|=1$, then $D=\{\lambda\}$ for some $\lambda \in \Lambda$. \\
We first use the Cauchy-Schwarz inequality and our assumption that $a_v=0$ for $v\not\in\Ss$.
\[
\Bigl| \sum_{v\in X} a_v\Bigr|^2 = \Biggl|\sum_{C\in \rho_\lambda(\Ss)}  \sum_{\substack{v\in X\\ \rho_\lambda(v)= C}} a_v \Biggr|^2 
\leq  \Bigl(\sum_{C\in\rho_\lambda(\Ss)}|C|\Bigr) \sum_{C\in \rho_\lambda(\Ss)}  \frac{1}{|C|} \Biggl|\sum_{\substack{v\in X \\ \rho_\lambda(v)=C}} a_v \Biggr|^2 \]
From (\ref{E:delta size}) we have $\sum_{C\in \rho_\lambda(\Ss)}|C|  = \mu_\lambda(\rho_\lambda(\Ss))|G_\lambda| \leq \delta_\lambda |G_\lambda|$,
so
\[
\Bigl| \sum_{v\in X} a_v\Bigr|^2 \leq \delta_\lambda |G_\lambda| \sum_{C\in G_\lambda^\sharp }  \frac{1}{|C|} \Biggl|\sum_{\substack{v\in X \\ \rho_\lambda(v)=C}} a_v \Biggr|^2.
\]
The characteristic function of a conjugacy class $C\in G_\lambda^\sharp $ in $G_\lambda$ has Fourier expansion
\[
\sum_{\chi\in\Irr(G_\lambda)} \bigg( \frac{1}{|G_\lambda|} \sum_{g\in C} \overline{\chi(g)} \bigg) \chi 
  =  \sum_{\chi\in\Irr(G_\lambda)}  \frac{|C|}{|G_\lambda|}  \overline{\chi(C)} \cdot \chi.
\]
We now substitute this into our previous inequality and expand.
\begin{align*} 
 \delta_\lambda^{-1} \Bigl| \sum_{v\in X} a_v\Bigr|^2 
& \leq  |G_\lambda| \sum_{C\in G_\lambda^\sharp } \frac{1}{|C|} \bigg|  \sum_{\substack{v\in X \\ \rho_\lambda(v)= C}} a_v \bigg|^2\\
&=  |G_\lambda| \sum_{C\in G_\lambda^\sharp } \frac{1}{|C|} \Biggl|  \sum_{{v\in X}} \Biggl(\sum_{\chi\in \Irr(G_\lambda)} \frac{|C|}{|G_\lambda|}  \overline{\chi(C)} \cdot                 \chi(\rho_\lambda(v)) \Biggr) a_v \Biggr|^2\\
&=   |G_\lambda| \sum_{C\in G_\lambda^\sharp } \frac{1}{|C|}  \sum_{{v,v'\in X}}    \sum_{\chi,\chi' \in \Irr(G_\lambda)} \frac{|C|^2}{|G_\lambda|^2}  \overline{\chi(C)} {\chi'(C)}               \chi(\rho_\lambda(v)) \overline{\chi'(\rho_\lambda(v'))} a_v \overline{a_{v'}}\\
&=  \sum_{\chi,\chi' \in \Irr(G_\lambda)} \Bigl(  \frac{1}{|G_\lambda|}   \sum_{C\in G_\lambda^\sharp } |C| \overline{\chi(C)} {\chi'(C)}        \Bigr)         \sum_{{v\in X}} a_v \chi(\rho_\lambda(v)) \overline{ \sum_{{v'\in X}}{a_{v'} \chi'(\rho_\lambda(v'))}  }\\
&=  \sum_{\chi,\chi' \in \Irr(G_\lambda)} (\chi,\chi')        \sum_{{v\in X}} a_v \chi(\rho_\lambda(v)) \overline{ \sum_{{v'\in X}}{a_{v'} \chi'(\rho_\lambda(v'))}  }
\end{align*}
Since the irreducible characters of $G_\lambda$ are orthonormal,
\begin{align*} 
\delta_\lambda^{-1} \Bigl| \sum_{v\in X} a_v\Bigr|^2
&\leq   \sum_{\chi \in \Irr(G_\lambda)}  \Bigl| \sum_{{v\in X}} a_v \chi(\rho_\lambda(v)) \Bigr|^2
=   \sum_{\chi\in \Irr(G_\lambda)-\{1\} } \Bigl| \sum_{v\in X} a_v \chi(\rho_\lambda(v)) \Bigr|^2 +     \Bigl| \sum_{v\in X}  a_v\Bigr|^2. 
\end{align*}
The lemma for $D=\{\lambda\}$ follows by noting that $\Prim(G_\lambda)=\Irr(G_\lambda)-\{1\}$ and collecting both sides; i.e.,
\[
\frac{1-\delta_\lambda}{\delta_\lambda} \Bigl| \sum_{v\in X} a_v\Bigr|^2  \leq \sum_{\chi\in \Prim(G_\lambda)} \Bigl| \sum_{v\in X} a_v \chi(\rho_\lambda(v)) \Bigr|^2.
\]
\noindent $\bullet$ Suppose that  $|D|\geq 2$.  Then $D=E\cup E'$, where $E$ and $E'$ are disjoint proper subsets of $D$.  We have a bijection
\begin{align*} 
\Irr(G_E) \times \Irr(G_{E'}) &\leftrightarrow \Irr(G_D), \; (\chi,\chi') \mapsto  \chi \chi',
\end{align*}
where $(\chi\chi')(g,g')=\chi(g)\chi(g')$ for $(g,g') \in G_{E}\times G_{E'} = G_D$.   This also induces a bijection between $\Prim(G_E) \times \Prim(G_{E'})$ and $\Prim(G_D)$.  Using the inductive hypothesis for $E$ and $E'$, we have:
\begin{align*} 
\sum_{\chi\in\Prim(G_D)}\Bigl| \sum_{v\in X} a_v \chi( \rho_D(v) ) \Bigr|^2 
&=  \sum_{\alpha\in\Prim(G_E)} \sum_{\beta \in\Prim(G_{E'})}\Bigl| \sum_{v\in X} a_v \alpha(\rho_E(v)) \beta(\rho_{E'}(v)) \Bigr|^2\\
&\geq   \Bigl( \prod_{\lambda\in E'} \frac{1-\delta_\lambda}{\delta_\lambda} \Bigr) \sum_{\alpha\in\Prim(G_E)} \Bigl| \sum_{v\in X} a_v \alpha(\rho_E(v)) \Bigr|^2\\
&\geq   \Bigl( \prod_{\lambda\in E'} \frac{1-\delta_\lambda}{\delta_\lambda} \Bigr)\Bigl( \prod_{\lambda\in E} \frac{1-\delta_\lambda}{\delta_\lambda}\Bigr) \Bigl| \sum_{v\in X} a_v  \Bigr|^2 
=  \Bigl( \prod_{\lambda\in D} \frac{1-\delta_\lambda}{\delta_\lambda} \Bigr) \Bigl| \sum_{v\in X} a_v  \Bigr|^2  \qedhere
\end{align*} 
\end{proof}

We now complete the proof of Theorem \ref{T:Large Sieve}.  Let $(a_v)_{v\in X}$ be a sequence of complex numbers with $a_v=0$ for $v\not\in \Ss$.  Using Lemma \ref{L:algebra lemma} and summing over all $D\in\Zz$ we obtain
\begin{align*}
\Bigl(\sum_{D\in \Zz} \prod_{\lambda\in D} \frac{1-\delta_\lambda}{\delta_\lambda} \Bigr) \Bigl|\sum_{v\in X} a_v\Bigr|^2  & \leq \sum_{D\in\Zz} \sum_{\chi\in\Prim(G_D)}\Bigl| \sum_{v\in X} a_v\chi(\rho_D(v))\Bigr|^2.
\end{align*}
The large sieve inequality (\ref{E:LSI of theorem}) then gives
\begin{equation} \label{E:smoothing}
\Bigl(\sum_{D\in \Zz} \prod_{\lambda\in D} \frac{1-\delta_\lambda}{\delta_\lambda} \Bigr) \Bigl|\sum_{v\in X} a_v\Bigr|^2 \leq \Delta(X,\rho,\Zz) \sum_{v\in X} |a_v|^2.
\end{equation}
In the special case where $a_v=1$ for $v\in\Ss$, we have
\[
\Bigl(\sum_{D\in \Zz} \prod_{\lambda\in D} \frac{1-\delta_\lambda}{\delta_\lambda} \Bigr)|\Ss|^2 \leq \Delta(X,\rho,\Zz) |\Ss|.
\]
The theorem follows by cancelling $|\Ss|$ from both sides (the theorem is trivial if $|\Ss|=0$).

\begin{remark}
Equation (\ref{E:smoothing}) can be useful in practice because it allows one to work with \emph{smoothed sums}.   We will not use this in the present paper.

\end{remark}
\subsection{Duality principle} 
\label{SS:Proof of LSI}
\begin{lemma}[Duality principle] \label{L:duality principle}
Let $I$ and $J$ be finite sets and let $\{c_{i,j}\}_{i \in I, \, j\in J}$ be a sequence of complex numbers.  Then the following assertions concerning a real number $\Delta$ are equivalent:
\begin{romanenum}
\item For any sequence $\{x_i\}_{i \in I}$ of complex numbers,
\[
\sum_{j\in J} \big| \sum_{i\in I} c_{i,j} x_i \big|^2  \leq \Delta  \sum_{i \in I} |x_i|^2.
\]
\item For any sequence $\{y_j\}_{j \in J}$ of complex numbers,
\[
 \sum_{i\in I} \big| \sum_{j\in J} c_{i,j} y_j \big|^2  \leq \Delta  \sum_{j \in J} |y_j|^2.
\]
\end{romanenum}
\end{lemma}
\begin{proof}
This is a special case of \cite{Montgomery}*{Lemma 2}.
\end{proof}

\begin{lemma} \label{L:duality principle 2}
Let $I$ and $J$ be finite sets and let $\{c_{i,j}\}_{i \in I, \, j\in J}$ be a sequence of complex numbers.    Then for any sequence $\{x_i\}_{i \in I}$ of complex numbers, we have
\[
 \sum_{j\in J} \big| \sum_{i\in I} c_{i,j} x_i \big|^2  \leq \bigg(\max_{j' \in J} \sum_{j \in J} \big| \sum_{i\in I} c_{i,j} \bbar{c_{i,j'}} \big|
 \bigg)\sum_{i \in I} |x_i|^2.
\]
\end{lemma}
\begin{proof}
Take any sequence $\{y_j\}_{j \in J}$ of complex numbers.
\begin{align*}
\sum_{i\in I} \big| \sum_{j\in J} c_{i,j} y_j \big|^2 & = \sum_{j,j' \in J} \sum_{i \in I} c_{i,j}\bbar{c_{i,j'}} y_j \bbar{y_{j'}} \\
& \leq\sum_{j, j'\in J} \Bigl| \sum_{i \in I}  c_{i,j} \bbar{c_{i,j'}} \Bigr| |y_j| |y_{j'}| \\
&\leq\sum_{j,\, j'} \Bigl| \sum_{i} c_{i,j} \bbar{c_{i,j'}} \Bigr| \frac{|y_j|^2 + |y_{j'}|^2}{2} \\
&=\sum_{j,\, j'} \Bigl| \sum_{i}  c_{i,j}\bbar{ c_{i,j'}} \Bigr|  |y_{j'}|^2 
\leq \Bigl( \max_{j'} \sum_j \Bigl| \sum_{i} c_{i,j}\bbar{ c_{i,j'}}\Bigr|\Bigr) \sum_{j'}  |y_{j'}|^2 
\end{align*}
The lemma is now an immediate consequence of Lemma~\ref{L:duality principle}.
\end{proof}

\begin{proof}[Proof of Proposition~\ref{P:LSI}]
Define $I:= X$ and $J:= \bigcup_{D \in \Zz} \Prim(G_D)$.  For a character $\chi\in J$, let $D_\chi$ be the element of $\Zz$ for which $\chi\in\Prim(G_{D_\chi})$.  For $v \in I$ and $\chi \in J$, define $c_{v,\chi}:=\chi(\rho_{D_\chi}(v))$.   The proposition then follows directly from Lemma~\ref{L:duality principle 2}.
\end{proof}

\begin{remark} 
Let $C=(c_{i,j})$ be an $m\times n$ matrix with complex entries.   For a (not necessarily prime!) value $p$ with $1\leq p \leq \infty$, we can endow $\CC^n$ with the usual $p$-norm $\norm{\cdot}_p$.  We then define $\norm{C}_p$ to be the supremum of $\norm{Cy}_p/\norm{y}_p$ over all non-zero $y\in \CC^n$.
In particular, note that $\norm{C}_2$ is the smallest nonnegative number such that 
\[
 \sum_{i = 1}^m \big| \sum_{j=1}^n c_{i,j} y_j \big|^2  \leq \norm{C}_2^2 \sum_{j =1}^n |y_j|^2
\]
holds for all $y \in \CC^n$.  Lemma~\ref{L:duality principle} is thus equivalent to $\norm{C}_2 = \norm{C^*}_2$ where $C^*$ is the conjugate transpose of $C$.
For a complex matrix $A$, one has $\norm{A}_\infty = \max_{i} \sum_j |a_{i,j}|$.  In particular,
\[
\norm{C^*C}_\infty = \max_{j'} \sum_{j} |\sum_{i} \overline{c_{i,j'}} c_{i,j} |.
\]
Lemma~\ref{L:duality principle 2} is thus equivalent to $\norm{C}_2^2 \leq \norm{C^*C}_\infty$.  
\end{remark}

\section{The large sieve for Galois representations} \label{S:Frobenius Sieve}
\subsection{Galois representations} 
\begin{defn}
Let $k$ be a number field and let $H$ be a (Hausdorff) topological group.  
A homomorphism $\rho\colon \calG_k \to H$ is a \defi{Galois representation} of $k$ 
if it is continuous (where $\calG_k=\Gal(\kbar/k)$ is endowed with the Krull topology). 
\end{defn}

Let $\rho\colon \calG_k\to H$ be a Galois representation.  The group $\ker(\rho)$ is a closed normal subgroup of $\calG_k$ whose fixed field we denote by $k(\rho)$.    The representation $\rho$ thus factors through an injective homomorphism $\Gal(k(\rho)/k)\hookrightarrow H$.   The representation $\rho$ is \defi{unramified} at a prime $\p\in\Sigma_k$ if $\p$ is unramified in the field extension $k(\rho)/k$.  

Fix a prime $\p \in \Sigma_k$ for which $\rho$ is unramified at $\p$ and take any place $\P$ of $k(\rho)$ which extends $\p$.  Since $\p$ is unramified in $k(\rho)$, we have a well-defined element $\rho(\Frob_\P) \in H$.  The conjugacy class of $\rho(\Frob_\P)$ in $\rho(\calG_k)$ does not depend on the choice of $\P$ and we shall denote it by $\rho(\Frob_\p)$. 

\begin{defn}
Let $\{\rho_\lambda\colon\calG_k \to H_\lambda \}_{\lambda\in \Lambda}$ be a collection of Galois representations.  We say that the representations $\{\rho_\lambda\}_{\lambda\in\Lambda}$ are \defi{independent} if 
\[
\Bigl(\prod_{\lambda\in\Lambda}\rho_\lambda \Bigr)(\calG_k) = \prod_{\lambda\in\Lambda} \rho_\lambda(\calG_k).
\]
An equivalent definition of independence is that the fields $k(\rho_\lambda)$ are linearly disjoint over $k$.
\end{defn}
\subsection{Statement of the large sieve}

\begin{thm} \label{T:LS for Frobenius}
Let $F$ be a number field and let $\Lambda$ be a set of nonzero ideals of $\OO_F$ which are pairwise relatively prime.  Let $k$ be a number field and suppose we have a collection of independent Galois representations
\[
\big\{\rho_\lambda \colon \calG_k \to H_\lambda \big\}_{\lambda\in\Lambda}.
\]
Assume that all the groups $G_\lambda:=\rho_\lambda(\calG_k)$ are finite and that there exists a real number $r\geq 1$ such that $|G_\lambda| \leq N(\lambda)^r$ for all but finitely many $\lambda \in \Lambda$.
Assume further that there is a finite set $S\subseteq \Sigma_k$ such that each $\rho_\lambda$ is unramified away from $S_\lambda:= S\cup  \{\p \in \Sigma_k : \p | N(\lambda)\}$.  \\
 
For every $\lambda\in \Lambda$, fix a non-empty subset $C_\lambda$ of $G_\lambda$ that is stable under conjugation.    Let $Q=Q(x)$ be a positive function of a real variable $x$ such that $Q(x) \ll  \sqrt{x}$ and let $\Lambda(Q)$ be the set of $\lambda\in \Lambda$ with $N(\lambda)\leq Q$.    Define the set
\begin{equation*} 
\Ss(x) := \big\{ \p \in \Sigma_k(x) : \; \p\in S_\lambda \;\text{ or }\; \rho_\lambda(\Frob_\p) \subseteq C_\lambda  \text{ for all $\lambda\in \Lambda(Q)$}\big\}.\\
\end{equation*} 

Choose subsets $\Zz(Q)\subseteq \{ D : D\subseteq \Lambda,\; {\prod}_{\lambda\in D} N(\lambda)\leq Q\}$ and define
\[
L(Q) = \sum_{D\in\Zz(Q)} \prod_{\lambda \in D} \frac{1-|C_\lambda|/|G_\lambda|}{|C_\lambda|/|G_\lambda|}.
\]
For each $D\subseteq \Lambda$, define $G_D=\prod_{\lambda\in D} G_\lambda$. 
\begin{romanenum}  
\item  Let $B>0$ be a real number.  If $Q(x) := c \big({\log x}/{(\log \log x)^2} \big)^{{1}/{(6r)}}$ for a sufficiently small constant $c>0$, then
\[  
	|\Ss(x)| \leq \Big(\Li x   + O(x/(\log x)^{1+B}) \Big) L(Q)^{-1}.
\]
\item 
Assuming the Generalized Riemann Hypothesis,
\[
	|\Ss(x)| \leq \Big(\Li x + O\bigl( \max_{D'\in\Zz(Q)}|G_{D'}|\cdot \sum_{D\in\Zz(Q)}|G_D^\sharp||G_D|\cdot x^{1/2}\log x\bigr) \Big) L(Q)^{-1}.
\]
\item 
Assuming Artin's Holomorphy Conjecture for the extensions $k(\rho_{D\cup D'})/k$ for $D,D'\in \Zz(Q)$ and assuming the Generalized Riemann Hypothesis,
\[
	 |\Ss(x)| \leq \Bigg(\Li x + O\Big( \max_{D'\in\Zz(Q),\, \chi' \in \Irr(G_{D'})}\chi'(1)\sum_{D\in\Zz(Q), \, \chi\in \Irr(G_D)}\chi(1) \cdot x^{1/2}\log x\Big) \Bigg)L(Q)^{-1}.
\]
\end{romanenum}
The implicit constants depend on $k$, the representations $\{\rho_\lambda\}_{\lambda\in \Lambda}$ and in part (i) also on $r$ and $B$.
\end{thm}

\begin{remark} \label{R:LS for Frobenius}
\begin{romanenum}
\item
If $L(Q)=0$, then one should interpret the theorem as giving the trivial bound $|\Ss(x)|\leq +\infty$.
\item 
The dependence of the bounds in Theorem~\ref{T:LS for Frobenius} on the sets $C_\lambda$ are only in terms of the ratios $|C_\lambda|/|G_\lambda|$.  For each $\lambda \in \Lambda$, fix a real number $0<\delta_\lambda\leq 1$ such that $|C_\lambda|/|G_\lambda| \leq \delta_\lambda$. Then
\[
L(Q)  \geq \sum_{D\in\Zz(Q)} \prod_{\lambda \in D} \frac{1-\delta_\lambda}{\delta_\lambda}.
\]

The smaller the values of $\delta_\lambda$ are, the stronger our bound on $|\Ss(x)|$ is.  The ``large'' in the large sieve refers to the fact that one may take $\delta_\lambda$ to be relatively small (at least smaller than earlier sieve methods); i.e., a \emph{large} number of elements $G_\lambda$ are not hit by the conjugacy classes $\rho_\lambda(\Frob_\p)$. 

The example of \S\ref{SS:toy example} is typical of a large sieve where we have $\delta_\ell = 1/2 + O(1/\ell)$.   In our application to the Koblitz conjecture, we will have $\delta_\ell=(1-1/\ell) + O(1/\ell^2)$ which is typical of so-called ``small sieves''.

\item There is flexibility in what the set $\Zz(Q)$ can be.   
The choice $\Zz(Q)=\{  D : D\subseteq \Lambda, \, \prod_{\lambda\in D} N(\lambda)\leq Q\}$ is usually appropriate, but as we will see other subtle choices may be useful.  
\item  Suppose that $s$ is a number such that $|G_\lambda^\sharp|\leq N(\lambda)^s$ for all but finitely many $\lambda\in \Lambda$ (one can always take $s=r$).  
Assuming GRH, the bound in Theorem~\ref{T:LS for Frobenius}(ii) gives the simpler expression
$|\Ss(x)| \leq (\Li x + O( Q^{2r+s+1} x^{1/2}\log x) )L(Q)^{-1}$.  Choosing $Q(x)= \big(x^{1/2}/(\log x)^2\big)^{1/(2r+s+1)}$, we obtain the bound 
\[
|\Ss(x)| \ll \frac{x/\log x}{L\big(x^{1/(4r+2s+2)}/(\log x)^{2/(2r+s+1)}\big)}.
\]
\item In many arithmetic situations (including those considered in this paper) we have $G_\lambda \subseteq \GG(\OO_F/\lambda)$ where $\GG$ is a group scheme of finite type over $\Spec \OO_F$.  In Theorem \ref{T:LS for Frobenius}, one can then take $r$ to be any value greater than the dimension of $\GG$.
\end{romanenum}
\end{remark}

\subsection{Abelian varieties}
We now recall some basic facts concerning Galois representations associated to abelian varieties;  these will supply us with interesting examples that satisfy the conditions of Theorem~\ref{T:LS for Frobenius}.  In particular, these representations will be needed for our applications.\\

Let $A$ be an abelian variety of dimension $g\geq 1$ defined over a number field $k$.    For each integer $m\geq 1$, the absolute Galois group of $k$ acts on the $m$-torsion points $A[m]$ of $A(\kbar)$ inducing a Galois representation
\[
\rho_{A,m} \colon \calG_k \to \Aut(A[m]) \cong \GL_{2g}(\ZZ/m\ZZ).
\]
Let $S_A$ be the set of $\p \in \Sigma_k$ for which $A$ has bad reduction.   The representation $\rho_{A,m}$ is unramified outside of $S_A\cup\{ \p\in \Sigma_k: \p | m\}$.   For every prime ideal $\p \in \Sigma_k - S_A$, there is a unique polynomial $P_{A_\p}(T) \in \ZZ[T]$ such that
\[
P_{A_\p}(T) \equiv \det(TI-\rho_{A,m}(\Frob_\p)) \bmod{m}
\]
for all positive $m$ with $\p \nmid m$.  This agrees with the definition of $P_{A_\p}(T)$ given in \S\ref{SSS:Abelian varieties over finite fields}.

Now fix a polarization $\phi\colon A\to A^\vee$.  Combining this polarization with the Weil pairing gives an alternating bilinear form $e_m\colon A[m] \times A[m] \to \mu_m$.  For $x,y \in A[m]$ and $\sigma \in \calG_k$, we have
\[
e_m(\sigma x, \sigma y) = \sigma( e_m(x,y) ) = e_m(x,y)^{\chi_{k,m}(\sigma)}
\]
where $\chi_{k,m}\colon \calG_k \to (\ZZ/m\ZZ)^\times$ is the cyclotomic character of $k$ modulo $m$.  

Let $\GSp(A[m], e_m)$ be the group of $C\in \Aut(E[m])$ for which there exists an $\mult(C)\in (\ZZ/m\ZZ)^\times$ such that $e_m(Cx,Cy)= e_m(x,y)^{\mult(C)}$ for all $x,y \in A[m]$.   Our Galois representation thus becomes
\[
\rho_{A,m} \colon \calG_k \to \GSp(A[m],e_m).
\]
If $m$ is relatively prime to the degree of $\phi$, then the pairing $e_m$ is non-degenerate.  In this case, the isomorphism class of the pair $(A[m],e_m)$ depends only on $g$ and $m$, and we denote the corresponding abstract group by $\GSp_{2g}(\ZZ/m\ZZ)$.\\

We now give some examples of abelian varieties for which we can describe the image of $\rho_{A,m}$. 

\begin{thm}[Serre] \label{T:Abelian variety with full GSp}
Let $A$ be an abelian variety of dimension $g$ defined over a number field $k$.  Suppose that the following hold:
\begin{romanenum}
\item $\End_{\kbar}(A)=\ZZ$,
\item $g$ is either $2$, $6$, or an odd integer.
\end{romanenum}
Then $\rho_{A,m}(\calG_k)$ is a subgroup of $\GSp(A[m],e_m)$ whose index is bounded independent of $m$.  In particular, there exists a positive integer $B$ such that  $\rho_{A,m}(\calG_k) \cong \GSp_{2g}(\ZZ/m\ZZ)$ for all $m$ relatively prime to $B$.
\end{thm}
\begin{proof}
In the case of non-CM elliptic curves (i.e., $g=1$), this is a well-known result of Serre \cite{Serre-Inv72} (we may choose $\phi$ to be a principal polarization and then $\GSp(A[m],e_m) \cong \GL_2(\ZZ/m\ZZ)$).

In the general case, Th\'eor\`eme 3 from the \emph{R\'esum\'e des cours de 1985-1986} in \cite{Serre-VolumeIV} implies that there is a $B$ such that $\rho_{A,\ell}$ is surjective for all $\ell\nmid B$.  
For an overview of the proof, see the letters at the beginning of \cite{Serre-VolumeIV} especially the one to Marie-France Vign\'eras.   The theorem then follows from the group theoretic Lemmas 1 and 2 in Serre's letter to Vign\'eras.
\end{proof}

The following describes the image of the Galois representations when $A$ is a product of non-CM elliptic curves that are pairwise non-isogenous.

\begin{thm} \label{T:Elliptic curves with large monodromy}
Let $E_1,\dots, E_n$ be elliptic curves without complex multiplication defined over a number field $k$.  Assume that the curves $E_i$ are pairwise non-isogenous over $\kbar$.  Then there exists a positive integer $B$ such that
\[
\rho_m:=\Big(\prod_{i=1}^n \rho_{E_i,m}\Big) (\calG_k) = \big\{ (A_i) \in \GL_2(\ZZ/m\ZZ)^n : \det(A_1)=\cdots=\det(A_n) \big\}
\]
for all $m$ relatively prime to $B$.
\end{thm}
\begin{proof}
This follows from \cite{Ribet-modular}*{Theorem~3.5} (from Faltings we know that Ribet's hypothesis is equivalent to the curves $E_i$ being pairwise non-isogenous over $\kbar$).
\end{proof}

\begin{remark}
\begin{romanenum}
\item Let $A$ be an abelian variety over a number field $k$.  In contrast to the abelian varieties occurring in Theorems \ref{T:Abelian variety with full GSp} and \ref{T:Elliptic curves with large monodromy}, there need not exist an integer $B$ for which the Galois representations $\{\rho_{A,\ell} \colon \calG_k \to \Aut(A[\ell]) \}_{\ell \nmid B}$ are independent.  However, there does exists a finite extension $K/k$ such that the representations $\{\rho_{A,m}\colon \calG_K \to \Aut(A[m])\}_{m\in \Lambda}$ are independent for any set $\Lambda$ of relatively prime natural numbers (see Serre's letter to Ribet \cite{Serre-VolumeIV}*{p56}).  
\item
Conjecturally many other systems of Galois representations will satisfy the conditions of Theorem~\ref{T:LS for Frobenius}.   See \cite{SerreMotivicGalois}*{\S2} for a discussion of $\ell$-adic representations associated to motives.
\end{romanenum}
\end{remark}

\subsection{Proof of Theorem \ref{T:LS for Frobenius}}

Fix notation as in Theorem~\ref{T:LS for Frobenius}. We first explain how to apply our abstract large sieve (Theorem~\ref{T:Large Sieve}).  For each finite set $D\subseteq \Lambda$, define the Galois representation
\[
\rho_D:= \Big(\prod_{\lambda\in D}\rho_\lambda\Big) \colon \calG_k \to \prod_{\lambda\in D} H_\lambda.
\] 
By our ramification assumptions, we find that $\rho_D$ is unramified at all $\p\in \Sigma_k(x)-S_D$ where $S_D:= S\cup \{ \p\in \Sigma_k: \p\mid \prod_{\lambda \in D} N(\lambda)\} = \bigcup_{\lambda \in D} S_\lambda$.
That the representations $\{\rho_\lambda\}_{\lambda\in\Lambda}$ are independent implies that $G_D= \prod_{\lambda\in D} G_\lambda$ is the image of $\rho_D$ and hence $\rho_D$ induces an isomorphism $\Gal(k(\rho_D)/k)\overset{\sim}{\to} G_D$.\\

Define $X=\Sigma_k(x)$.   For each $\lambda\in \Lambda(Q)$, fix a function $\rho_\lambda \colon X \to G_\lambda^\sharp$ such that the following conditions hold:
\begin{itemize}
\item
$\rho_\lambda(\p) = \rho_\lambda(\Frob_\p)$ \; if $\p \in \Sigma_k(x)-S_\lambda$,
\item
$\rho_\lambda(\p) \subseteq C_\lambda$ \; \quad\quad\quad if $\p \in \Sigma_k(x) \cap S_\lambda$ 
\end{itemize}
(the second condition is imposed simply to match the set-up of our abstract large sieve).  It will be clear from context which function $\rho_\lambda$ we are using.  
 
 For $D\subseteq \Lambda(Q)$, define the function $\rho_D:= \big(\prod_{\lambda\in D} \rho_\lambda\big) \colon X \to G_D^\sharp$.   We may now define sets of primitive characters $\Prim(G_D)$ just as in \S\ref{S:LS setup}.  

 For every $\lambda\in\Lambda(Q)$, define $\mu_\lambda$ to be the measure on $G_\lambda$ as in \S\ref{S:LS setup}.   We have $\rho_\lambda(\p)\subseteq C_\lambda$ for all $\p \in \Ss(x)$ and hence
\[ 
\mu_\lambda(\rho_\lambda(\Ss(x))) \leq {|C_\lambda|}/{|G_\lambda|}.
\]
We may now apply our abstract large (Theorem~\ref{T:Large Sieve}), which gives
\begin{equation} \label{E:ultimate LSI}
L(Q)|\Ss(x)| \leq \Delta(X,\rho,\Zz(Q))
\end{equation}
where $L(Q) = \sum_{D\in\Zz(Q)} \prod_{\lambda \in D} \frac{1-|C_\lambda|/|G_\lambda|}{|C_\lambda|/|G_\lambda|}$.
It remains to bound $\Delta(X,\rho,\Zz(Q))$.  From Proposition \ref{P:LSI} we have
 \begin{equation} \label{E:Duality bound for Frobenius}
 \Delta(X,\rho,\Zz(Q)) \leq \max_{\substack{D'\in \Zz(Q) \\ \chi'\in \Prim(G_{D'})} }  \sum_{D\in\Zz(Q)} \sum_{\chi\in\Prim(G_D)}  \Bigl| \sum_{\p\in \Sigma_k(x)} \chi(\rho_{D}(\p)) \overline{\chi'(\rho_{D'}(\p))} \Bigr|,
 \end{equation} 
and to bound this quantity we will make use of the character sums worked out in Appendix~\ref{A:Character sums}.

\begin{lemma} \label{L:gory lemma}
With assumptions as above; fix $D,D' \in \Zz(Q)$ and characters $\chi\in \Irr(G_D)$, $\chi'\in \Irr(G_{D'})$.
\begin{romanenum}
\item 
Let $B>0$ be a constant.  If $Q(x) := c \big({\log x}/{(\log \log x)^2} \big)^{{1}/{(6r)}}$ for a constant $c>0$ sufficiently small, then
\[
 \sum_{\p\in \Sigma_k(x)} \chi(\rho_D(\p)) \overline{\chi'(\rho_{D'}(\p))}  =  \delta_{\chi,\chi'}\Li x + O\Bigl( \frac{x}{(\log x)^{1+B}} \Bigr). \text{\quad\quad\;\;}
\]
\item Assuming GRH, 
\[
\sum_{\p\in \Sigma_k(x)} \chi(\rho_D(\p)) \overline{\chi'(\rho_{D'}(\p))}  =  \delta_{\chi,\chi'} \Li x + O\Bigl( |G_D||G_{D'}|x^{1/2} \log x \Bigr). 
\]
\item Assuming AHC for the extension $k(\rho_{D\cup D'})/k$ and assuming GRH, 
\[
\sum_{\p\in \Sigma_k(x)} \chi(\rho_D(\p)) \overline{\chi'(\rho_{D'}(\p))}  =  \delta_{\chi,\chi'}\Li x + O\Bigl( \chi(1)\chi'(1) x^{1/2} \log x \Bigr).
\]
\end{romanenum}
\end{lemma}
\begin{proof}
To ease notation, define $T:=\sum_{\p\in \Sigma_k(x)} \chi(\rho_D(\p)) \overline{\chi'(\rho_{D'}(\p))}$ and $L:=k(\rho_{D\cup D'})$.  We may view $\chi$ (resp.~$\chi'$) as an irreducible character of $G_{D \cup D'}$ by composing with the projection maps from $G_{D\cup D'}$ to $G_D$ (resp.~$G_{D'}$).

The Galois representation $\rho_{D \cup D'}$ is unramified at all $\p \not\in S_{D \cup D'}:= S \cup \{ \p\in \Sigma_k : \p | \prod_{\lambda\in {D \cup D'}} N(\lambda)\}$.   In particular, 
$\rho_D(\p)=\rho_D(\Fr_\p)$ and $\rho_{D'}(\p)=\rho_{D'}(\Fr_\p)$ for all $\p \in \Sigma_k(x) - S_{D\cup D'}$.   Since $D,D' \in \Zz(Q)$ and $Q(x)\ll \sqrt{x}$, we have
\begin{align*}
|S_{D\cup D'}| \ll |S| + \sum_{\lambda \in D\cup D'} \log N(\lambda) &\leq |S|+  \log({\prod}_{\lambda\in D} N(\lambda)\cdot {\prod}_{\lambda\in D'} N(\lambda) )\\
& \leq |S| + \log (Q \cdot Q) \ll \log x.
\end{align*}
Therefore,
\begin{align} \label{E:reduce to character sum}
T & =\sum_{ \substack{ \p\in \Sigma_k(x)\\  \text{unramified in }L} }(\chi\overline{\chi'})(\rho_{D\cup D'}(\Frob_\p))
+ O(\chi(1)\chi'(1)|S_{D\cup D'}|  )\\
\notag 
& =\sum_{ \substack{ \p\in \Sigma_k(x)\\  \text{unramified in }L} }(\chi\overline{\chi'})(\rho_{D\cup D'}(\Frob_\p))
+ O(\chi(1)\chi'(1) \log x).
\end{align}

Before considering the different cases, we first bound some quantities that will show up in our character sums.   For any $D\in \Zz(Q)$, we have $|G_D| \ll \prod_{\lambda\in D} N(\lambda)^r \leq Q^r$.  Thus
\[
[L:k] \leq |G_D||G_{D'}| \ll Q^{2r}.
\]
Let $P(L/k)$ be the set of primes $p$ for which there exists a $\p \in \Sigma_k$ such that $\p|p$ and $\p$ is ramified in $L$.  From our ramification assumptions,
\[
\prod_{p\in P(L/k)}p \leq  \prod_{\lambda\in S_{D\cup D'}} N(\lambda) \ll \prod_{\lambda\in D} N(\lambda) \prod_{\lambda\in  D'} N(\lambda) \leq Q^2.
\]
Therefore
\begin{align*} 
M(L/k)&:=[L:k] d_k^{1/[k:\QQ]}\prod_{p\in P(L/k)} p\ll Q^{2r+2}.
\end{align*}
In particular, $\log M(L/k) \ll \log x$, since we have assumed that $Q(x)\ll \sqrt{x}$.\\

\noindent(iii) Assume AHC and GRH.  Applying Proposition~\ref{P:conditional character sum}(ii) to (\ref{E:reduce to character sum}), we have
\begin{align*} 
T &= (\chi\overline{\chi'},1) \Li x + O\Bigl( \chi(1)\chi'(1) [k:\QQ] x^{1/2} \log(M(L/k) x) \Bigr)\\
&= \delta_{\chi,\chi'} \Li x + O\Bigl( \chi(1)\chi'(1) x^{1/2} \log  x \Bigr).
\end{align*} 

\noindent(ii) Assume GRH.  Applying Proposition~\ref{P:conditional character sum}(i) to (\ref{E:reduce to character sum}), we have
\begin{align*} 
T &= (\chi\overline{\chi'},1) \Li x + O\Bigl( \Bigl({\sum}_{g\in G_{D\cup D'}}|\chi(g)||\chi'(g)| \Bigr) [k:\QQ] x^{1/2} \log(M(L/k) x) \Bigr)\\
&= \delta_{\chi,\chi'} \Li x + O\Bigl( {\sum}_{g\in G_{D\cup D'}}|\chi(g)||\chi'(g)|  x^{1/2} \log  x \Bigr).
\end{align*}
By the Cauchy-Schwartz inequality and the irreducibility of the characters $\chi$ and $\chi'$,
\[
\sum_{g\in G_{D\cup D'}}|\chi(g)||\chi'(g)| \leq \Big({\sum}_{g\in G_{D\cup D'}}|\chi(g)|^2\Big)^{1/2} \Bigl({\sum}_{g\in G_{D\cup D'}}|\chi'(g)|^2\Bigr)^{1/2} = |G_{D\cup D'}|.
\]
Therefore,
\[
T =\delta_{\chi,\chi'}\Li x + O\Bigl( |G_{D\cup D'}| x^{1/2} \log x \Bigr) = \delta_{\chi,\chi'}\Li x + O\Bigl( |G_D||G_{D'}|  x^{1/2} \log x \Bigr).
\]

\noindent(i) By Lemma \ref{L:discriminant bound}, 
$\log d_{L} \leq [L:\QQ] \log M(L/k) \ll Q^{2r} \log Q$,
and hence $10 [L:\QQ] (\log d_{L})^2\ll Q^{6r} (\log Q)^2$.  So for $Q(x):=c\Bigl(\frac{\log x}{(\log\log x)^2} \Bigr)^{\frac{1}{6r}}$, with a sufficiently small constant $c>0$, we will have $\log x \geq 10 [L:\QQ] (\log d_{L})^2$ for $x$ sufficiently large.  Applying Proposition~\ref{P:unconditional character sum} to (\ref{E:reduce to character sum}), we have
\[
   T - \delta_{\chi,\chi'}\Li x \ll  \chi(1)\chi'(1)\Li(x^{\beta_{L}})  + \chi(1)\chi'(1) |G_{D\cup D'}^\sharp |x\exp\Bigl( -c_1  [L:\QQ]^{-1/2}(\log x)^{1/2}\Bigr), \\
\]
where the $ \chi(1)\chi'(1)\Li(x^{\beta_{L}})$ term is present only when the exceptional zero $\beta_L$ exists.  The trivial bounds $|G_{D\cup D'}^\sharp |\leq |G_{D\cup D'}|\ll Q^{2r}$, $\chi(1)\leq |G_{D}|^{1/2}\ll Q^{r/2}$, and $\chi'(1)\leq |G_{D'}|^{1/2}\ll Q^{r/2}$ give
\[
T - \delta_{\chi,\chi'}\Li x \ll  Q^r\Li(x^{\beta_{L}})  + Q^{3r} x\exp\Bigl( -\frac{c_1'}{[k:\QQ]^{1/2}Q^r}  (\log x)^{1/2} \Bigr),
\]
for some constant $c'>0$.  The second term is easily bounded:
\begin{align*}
& Q^{3r} x\exp\Bigl( -\frac{c_1'}{[k:\QQ]^{1/2}Q^r}  (\log x)^{1/2} \Bigr)\\
\leq &  c^{3r} (\log x)^{1/2} x\exp\Bigl( -\frac{c_1'}{c^r [k:\QQ]^{1/2}}  (\log x)^{1/3}(\log\log x)^{1/3}\Bigr) \ll_{B}  \frac{x}{(\log x)^{1+B}}.
\end{align*}
Finally, consider the term containing the exceptional zero.  We have $Q^r\Li(x^{\beta_{L}}) \ll x^{\beta_{L}}$ so it suffices to show that $x^{\beta_{L}}\ll x/(\log x)^{1+B}$.  By Proposition~\ref{P:exceptional zeros}(ii), there is a field $F$ with $k\subseteq F\subseteq L$ such that $[F:k]\leq 2$ and $\zeta_F(\beta_L)=0$.  By Proposition~\ref{P:exceptional zeros}(iii),
\begin{align*}
1-\beta_L & \gg \min\big\{ ([F:\QQ]!\log d_F)^{-1}, d_F^{-1/[F:\QQ]}\big\}\\
&\geq \min\big\{ \big((2[k:\QQ]-1)!\log d_F^{1/[F:\QQ]}\big)^{-1}, d_F^{-1/[F:\QQ]}\big\} \gg d_F^{-1/[F:\QQ]}.
\end{align*}
By Lemma~\ref{L:discriminant bound}, $d_F^{1/[F:\QQ]} \leq M(F/k)$.  Using $P(F/k)\subseteq P(L/k)$, 
\[
d_F^{1/[F:\QQ]} \ll \prod_{p\in P(L/k)} p \ll Q^2 \ll (\log x)^{1/(3r)},
\]
and hence $1-\beta_L \gg (\log x)^{-1/(3r)}$.  Thus for $x$ sufficiently large, $(1-\beta_L) \log x \geq (1+B) \log\log x$, or equivalently $x^{\beta_L} \leq x/(\log x)^{1+B}$.
\end{proof}

We now bound $\Delta(X,\rho,\Zz(Q))$.  Theorem \ref{T:LS for Frobenius} will follow by combining these bounds with (\ref{E:ultimate LSI}).

\noindent (i) Fix any $B>1/3$.  By (\ref{E:Duality bound for Frobenius}) and Lemma \ref{L:gory lemma}(i),
\begin{align*} 
\Delta(X,\rho,\Zz(Q))&\leq  \max_{\substack{D'\in \Zz(Q) \\ \chi'\in \Prim(G_{D'})} }  \sum_{D\in\Zz(Q)} \sum_{\chi\in\Prim(G_D)}  \Bigl| \sum_{\p\in X} \chi(\rho_D(\p)) \overline{\chi'(\rho_{D'}(\p))} \Bigr| \\
&=  \max_{\substack{D'\in \Zz(Q) \\ \chi'\in \Prim(G_{D'})} }   \sum_{D\in\Zz(Q)} \sum_{\chi\in\Prim(G_D)} \Bigl( \delta_{\chi,\chi'} \Li x + O\Bigl( \frac{x}{(\log x)^{1+B}} \Bigr)  \Bigr)\\
&= \Li x + O\Bigl( \sum_{D\in\Zz(Q)} |\Prim(G_D)|  \frac{x}{(\log x)^{1+B}}  \Bigr).
\intertext{Now use the bound $\sum_{D\in\Zz(Q)}|\Prim(G_D)| \leq \sum_{D\in\Zz(Q)} |G_D| \ll Q^{r+1} \ll (\log x)^{1/3}$:}
\Delta(X,\rho,\Zz(Q)) &\leq \Li x + O\Bigl( \frac{x}{(\log x)^{1+(B-1/3)}} \Bigr).
\end{align*}

\noindent (ii) Assume GRH.  By (\ref{E:Duality bound for Frobenius}) and Lemma \ref{L:gory lemma}(ii),
\begin{align*} 
\Delta(X,\rho,\Zz(Q))&\leq  \max_{\substack{D'\in \Zz(Q)\\ \chi'\in \Prim(G_{D'})} }  \sum_{D\in\Zz(Q)} \sum_{\chi\in\Prim(G_D)}  \Bigl| \sum_{\p\in X} \chi(\rho_D(\p)) \overline{\chi'(\rho_{D'}(\p))} \Bigr| \\
&=  \max_{\substack{D'\in \Zz(Q) \\ \chi'\in \Prim(G_{D'})} }   \sum_{D\in\Zz(Q)} \sum_{\chi\in\Prim(G_D)} \Bigl( \delta_{\chi,\chi'} \Li x + O( |G_D||G_{D'}| x^{1/2} \log x ) \Bigr)\\
&= \Li x + O\Bigl(\max_{D'\in \Zz(Q)} |G_{D'}| \cdot \sum_{D\in\Zz(Q)} |\Prim(G_D)||G_D|  \cdot  x^{1/2}\log x \Bigr),
\end{align*}
and then use the inequality $|\Prim(G_D)| \leq |\Irr(G_D)| = |G_D^\sharp |$.

\noindent (iii) Assume AHC and GRH.  By (\ref{E:Duality bound for Frobenius}) and Lemma \ref{L:gory lemma}(iii),
\begin{align*} 
\Delta(X,\rho,\Zz(Q))&\leq  \max_{\substack{D'\in \Zz(Q) \\ \chi'\in \Prim(G_{D'})} }  \sum_{D\in\Zz(Q)} \sum_{\chi\in\Prim(G_D)}  \Bigl| \sum_{\p\in X} \chi(\rho_D(\p)) \overline{\chi'(\rho_{D'}(\p))} \Bigr| \\
&= \max_{\substack{D'\in \Zz(Q) \\ \chi'\in \Prim(G_{D'})} }   \sum_{D\in\Zz(Q)} \sum_{\chi\in\Prim(G_D)} \Bigl( \delta_{\chi,\chi'} \Li x + O( \chi(1)\chi'(1) x^{1/2} \log x ) \Bigr)\\
&= \Li x + 
O\Bigl(\max_{\substack{D'\in \Zz(Q) \\ \chi'\in\Irr(G_{D'})}}\chi'(1) \sum_{D\in\Zz(Q),\chi\in\Irr(G_D)} \chi(1)  \cdot  x^{1/2}\log x \Bigr).
\end{align*}

\section{The Koblitz Conjecture} \label{S:Koblitz}
The purpose of this section is to prove Theorem~\ref{T:Koblitz bound}.
We maintain the notation introduced in \S\ref{SS:Koblitz introduction}.  For each integer $m\geq 1$, let $G_m$ be the image of $\rho_{E,m}\colon \calG_k\to \GL_2(\ZZ/m\ZZ)$.  By Theorem~\ref{T:Abelian variety with full GSp}, there exists a positive integer $M$ such that the following conditions hold:
\begin{itemize}
\item $(\rho_{E,M}\times \prod_{\ell \nmid M} \rho_{E,\ell})(\calG_k) = G_M \times \prod_{\ell\nmid M} \GL_2(\ZZ/\ell\ZZ),$
\item  if $\ell$ divides $t_{E,k}$, then $v_\ell(t_{E,k}) < v_\ell(M)$ where $v_\ell$ is the $\ell$-adic valuation.
\end{itemize}
\subsection{Sieve setup}  
Fix a positive function $Q=Q(x)$ with $Q(x)\ll \sqrt{x}$; we will make a specific choice later.   We will bound the cardinality of the set
\[
\calS(x):=\{ \p \in \Sigma_k-S_E: ((t_{E,k}Q(x))^{1/2}+1)^2 < N(\p) \leq x,\; |E_\p(\FF_\p)|/t_{E,k} \text{ is prime} \}.
\]
Note that $P_{E,k}(x) = |\calS(x)| + O(Q) = |\calS(x)| + O(\sqrt{x}),$ so it suffices to bound $|\calS(x)|$.

For each $\p\in \calS(x)$, the Hasse bound gives
\begin{align*} 
|E(\FF_\p)|/t_{E,k} &\geq (N(\p)-2 N(\p)^{1/2}+1)/t_{E,k} = (N(\p)^{1/2}-1)^2/t_{E,k}  > Q(x),
\end{align*}
so the primality of $|E_\p(\FF_\p)|/t_{E,k}$ implies that 
\begin{equation} \label{E:Koblitz congruence}
|E_\p(\FF_\p)| \bmod{m} \in t_{E,k} (\ZZ/m\ZZ)^\times
\end{equation}
 for $m\leq Q$.  
 
 Define the set $\Lambda=\{M\} \cup \{\ell: \ell\nmid M\}$ and let $\Lambda(Q)$ be the set of $m\in\Lambda$ with $m\leq Q$.  By our choice of $M$, the Galois representations $\{\rho_{E,m}\}_{m\in \Lambda}$ are independent.   For $m\in\Lambda(Q)$ and $\p\in \calS(x)$, either $\p| m$ or $\det(I-\rho_{E,m}(\Frob_\p)) \equiv |E_\p(\FF_\p)| \bmod{m}$ and hence by (\ref{E:Koblitz congruence})
\[
\rho_{E,m}(\Frob_\p) \subseteq C_m := \{ A \in G_m : \det(I-A) \in t_{E,k} \left(\ZZ/m\ZZ\right)^\times \}.
\]
With our setup matching that of Theorem~\ref{T:LS for Frobenius}, we define 
\[
\Ss(x) := \big\{ \p \in \Sigma_k(x) - S_E : \p| m  \; \text{ or }\; \rho_{E,m}(\Frob_\p) \subseteq C_m  \text{ for all $m\in \Lambda(Q)$}\big\}.
\]
Note that $\calS(x) \subseteq \Ss(x)$, so it suffices to find upper bounds for $|\Ss(x)|$.

Define 
\[
\Zz(Q)=\{ D: D \subseteq \Lambda(Q),\, {\prod}_{m\in D} m \leq Q\} \text{\quad and \quad} L(Q)=\sum_{ D \in \Zz(Q) } \prod_{m\in D} \frac{1-|C_m|/|G_m|}{|C_m|/|G_m|}.
\]
For $D\in\Zz(Q)$, define $G_D=\prod_{m\in D} G_m$.  Before applying the large sieve, we will first carefully consider the asymptotics of $L(Q)$ as a function of $Q$.

\subsection{Asymptotics of $L(Q)$} \label{SS:Koblitz L}
\begin{lemma} \label{L:Koblitz factor}
Suppose $\ell$ is a prime such that $G_\ell = \GL_2(\ZZ/\ell\ZZ)$ and $\ell \nmid t_{E,k}$.  Then 
\[
\frac{1-|C_\ell|/|G_\ell|}{|C_\ell|/|G_\ell|}= \frac{1}{\ell} + \frac{2\ell^2 - \ell - 3}{\ell^4 - 2\ell^3 - \ell^2 + 3\ell}.
\]
\end{lemma}
\begin{proof}
In this case, $|C_\ell|$ is the number of  matrices $A\in\GL_2(\FF_\ell)$ that do not have $1$ as an eigenvalue; this can be counted directly, for example by using \cite{LangAlgebra}*{XVIII Table 12.4}.  We find that $|C_\ell|=\ell^4 - 2\ell^3 - \ell^2 + 3\ell$ and $|\GL_2(\FF_\ell)|=\ell(\ell-1)^2(\ell+1)$.  The lemma is now a direct computation.
\end{proof}
We introduce the Dirichlet series
\[
h(s) = \prod_{m\in \Lambda} \left( 1 + \frac{1-|C_m|/|G_m|}{|C_m|/|G_m|} m^{-s} \right) = \sum_{n=1}^\infty b_n/n^s;
\]
the significance is that $L(Q)=\sum_{n\leq Q} b_n$.  For each prime $\ell$, define 
\[
c_\ell=(2\ell^2 - \ell - 3)/(\ell^4 - 2\ell^3 - \ell^2 + 3\ell).
\]
Instead of $h(s)$, it will be more convenient to work with the Dirichlet series
\begin{equation} \label{E:L(s) equation}
g(s) := \prod_{\ell} \left( 1 + (1/\ell+c_\ell) \ell^{-s} \right)
 = \frac{\prod_{\ell|M} (1 + (\frac{1}{\ell}+c_\ell) \ell^{-s})}{1 + \frac{1-|C_M|/|G_M|}{|C_M|/|G_M|} M^{-s}} \cdot h(s),
\end{equation}
where the expression in terms of $h(s)$ follows from Lemma~\ref{L:Koblitz factor}.  The Dirichlet series $g(s)$ converges to a non-vanishing holomorphic function on the domain $\Re(s) > 0$.
\begin{lemma} \label{L:analytic continuation of g}
The function $g(s)$ has an analytic continuation to a nonvanishing function on a neighbourhood of $\Re(s)\geq 0$ except for a simple pole at $s=0$.  The residue of $g(s)$ at $s=0$ is $\prod_\ell \left( (1 + (\frac{1}{\ell}+c_\ell) ) (1-\frac{1}{\ell})    \right)$. 
\end{lemma}
\begin{proof}
For $\Re(s) > 0$,
\begin{align*} 
g(s)\zeta(s+1)^{-1} &= \prod_\ell \left( (1 + \Bigl(\frac{1}{\ell}+c_\ell) \ell^{-s}\Bigr) \Bigl(1-\ell^{-s-1}\Bigr)    \right)
= \prod_\ell ( 1 + c_\ell \ell^{-s} - c_\ell \ell^{-2s-1} - \ell^{-2s-2}).
\end{align*}
Since $c_\ell=2/\ell^2 + O(1/\ell^3)$, this Euler product converges absolutely and is nonvanishing in a neighbourhood of $\Re(s)\geq 0$.  The first statement of the lemma is now immediate.  The second statement follows by setting $s=0$ in the product and noting that $\zeta(s+1)$ has a simple pole at $s=0$ with residue $1$.
\end{proof}
By Lemma \ref{L:analytic continuation of g} and (\ref{E:L(s) equation}), we find that $h(s)$ analytically continues to a neighbourhood of $\Re(s) \geq 0$, except for a simple pole at $s=0$.  The residue of $h(s)$ at $s=0$ is
\begin{align*} 
 &  \left(1 + \frac{1-|C_M|/|G_M|}{|C_M|/|G_M|} \right)\prod_{\ell|M} \left(1 - \frac{1}{\ell}\right) \prod_{\ell\nmid M} \left( \left(1 + \frac{1-|C_\ell|/|G_\ell|}{|C_\ell|/|G_\ell|}\right)\left(1-\frac{1}{\ell}\right) \right) \\
&=  \Bigg(  \frac{|C_M|/|G_M|}{\prod_{\ell|M} \left(1 - \frac{1}{\ell}\right)} \prod_{\ell\nmid M} \frac{|C_\ell|/|G_\ell|}{\left(1-\frac{1}{\ell}\right)}\Bigg)^{-1} =: (C_{E,k})^{-1}.
\end{align*}

\begin{remark}
The number $C_{E,k}$ just introduced is exactly the constant from Conjecture \ref{C:Koblitz} that is predicted in \cite{Zywina-Koblitz}.  Using our assumptions on the integer $M$, it is easy to check that $C_{E,k}$ is independent of the initial choice of $M$.
\end{remark}

Applying the Wiener-Ikehara theorem \cite{Lang-ANT}*{XV~Theorem 1} to the Dirichlet series $h(s-1)$, which has a simple pole at $s=1$, we find that $\sum_{n\leq Q} nb_n = C_{E,k}^{-1} Q + o(Q).$  By partial summation (\cite{Murty-ANT}*{Theorem~2.1.1}), 
\begin{equation} \label{E:Koblitz L asymptotics}
L(Q)=\sum_{n\leq Q} b_n = C_{E,k}^{-1} \log Q + o(\log Q).
\end{equation}
\subsection{Proof of Theorem~\ref{T:Koblitz bound}}
We finally apply the large sieve.  First consider the unconditional case.  For all $m\in \Lambda$, $|G_m| \leq |\GL_2(\ZZ/m\ZZ)| \leq m^4$; so set $r=4$.  By Theorem~\ref{T:LS for Frobenius}(i), with $Q(x):= c\left({\log x}/{(\log\log x)^2}\right)^{1/24}$ we have
\[
|\Ss(x)| \leq (x/\log x + o(x/\log x))/L(Q).
\]  
From (\ref{E:Koblitz L asymptotics}),
\[
L(Q) = C_{E,k}^{-1}\log Q+o(\log Q)  = (24C_{E,k})^{-1}\log\log x +o(\log\log x).
\]
Therefore, 
\[ 
|\Ss(x)| \leq (24+o(1))C_{E,k}\frac{x}{ (\log x)(\log\log x)}.
\] 
Now assume GRH.  For $D\in\Zz(Q)$, we have $|G_D|\leq \prod_{m\in D} m^4 \leq Q^4$.  By Lemma~\ref{L:misc bounds for GL}, $|\GL_2(\ZZ/\ell\ZZ)^\sharp|  = \ell^2-1 \leq \ell^2$ and thus for $D\in \Zz(Q)$, $|G_D^\sharp| \ll Q^2$.
\[
\sum_{D \in \Zz(Q)} |G_D^\sharp||G_D| \ll \sum_{d\leq Q} Q^6 \leq Q^7
\]
By Theorem~\ref{T:LS for Frobenius}(ii) and (\ref{E:Koblitz L asymptotics}), 
\[
|\Ss(x)| \leq (\Li x + O(Q^{11}x^{1/2}\log x ))/L(Q) \leq (C_{E,k}+o(1))(\Li x + O(Q^{11}x^{1/2}\log x ))/\log Q.
\]
For a fixed constant $\delta>0$, define $Q(x):=x^{1/22}/(\log x)^{(2+\delta)/11}$.  Therefore
\[
|\Ss(x)| \leq (\Li x  + O(x/(\log x)^{1+\delta}) )/ (  (22C_{E,k})^{-1} \log x + o(\log x)  )  \leq (22 + o(1) ) C_{E,k} \frac{x}{(\log x)^2}.
\]

\section{Elliptic curves and thin sets}
The goal of this section is to prove Theorem~\ref{T:LS for elliptic curves}. 
\subsection{Reduction of thin sets}
 \begin{defn}  Let $\Omega$ be a subset of $\ZZ^n$. For each prime $\ell$, let $\Omega_\ell \subseteq (\ZZ/\ell\ZZ)^n$ be the reduction of $\Omega$  modulo $\ell$.
 \end{defn}
\begin{lemma}\label{L:reduction of thin sets}
Let $\Omega\subseteq \ZZ^n$ be a thin set.  
\begin{romanenum}
\item There are thin sets $\Omega_1,\ldots,\Omega_m\subseteq \ZZ^n$, a set of primes $\Lambda\subseteq \Sigma_\QQ$ with positive natural density, and a real number $0<c<1$ such that $\Omega=\bigcup_{i=1}^m\Omega_i$,
and $|\Omega_{i,\ell}| \leq c \ell^n$ for all $1\leq i\leq m$ and $\ell \in \Lambda$.

\item Suppose $\Omega$ is a thin set of Type 1.  Then $|\Omega_\ell| \ll \ell^{n-1}$ for all $\ell \in \Sigma_\QQ$,
where the implied constant depends on $\Omega$ and $n$.
\end{romanenum} 
\end{lemma}
\begin{proof}
Part (i) is a consequence of \cite{SerreMordellWeil}*{\S13~Theorem 5}.  Part (ii) follows from the Lang-Weil bounds \cite{LangWeil}. 
\end{proof}

\subsection{Proof of Theorem~\ref{T:LS for elliptic curves}}
We are interested in bounding the cardinality of the set
$\calS(x):=\{\p \in \Sigma_k(x) :  (a_\p(E_1),\ldots,a_\p(E_n),N(\p))\in \Omega \}$.  By Lemma~\ref{L:reduction of thin sets}, we need only consider the case where there is a set $\Lambda\subseteq \Sigma_k$ of positive density and a number $0<c<1$ such that $|\Omega_\ell| \leq c \ell^{n+1}$ for all $\ell\in \Lambda$.  

Let $\GG/\Spec\ZZ$ be the algebraic subgroup of $(\GL_2)^n$ such that
\[
\GG(R) = \{ (A_1,\dots, A_n) \in \GL_2(R)^n : \det(A_1) = \cdots = \det(A_n) \}
\]
for each commutative ring $R$.   For every integer $m\geq 1$, we have a Galois representation
\[
\rho_m:= \prod_{i=1}^n \rho_{E_i,m}  \colon \calG_k \to \GL_2(\ZZ/m\ZZ)^n.
\]
By Theorem~\ref{T:Elliptic curves with large monodromy}, there is an integer $B$ such that $\rho_m(\calG_k) = \GG(\ZZ/m\ZZ)$ for all $m$ relatively prime to $B$.   We may assume that $\ell \nmid B$ for all $\ell \in \Lambda$ and hence the Galois representations $\{\rho_\ell\}_{\ell \in \Lambda}$ are independent.  Let $S$ be a finite subset of $\Sigma_k$ such that the elliptic curves $E_1,\ldots,E_n$ have good reduction outside $S$.  

For each prime $\ell \in \Lambda$, define the set
\[
C_\ell= \{ (A_i) \in \GG(\ZZ/\ell\ZZ) : (\tr(A_1),\ldots,\tr(A_n),\det(A_1)) \in \Omega_\ell\}.
\]
For all $\ell\in\Lambda$ and $\p\in \calS(x)$, either $\p \in S_{\ell}:= S \cup \{ \p : \p |\ell
\}$ or $\rho_\ell(\Frob_\p) \subseteq C_\ell.$

Fix a positive function $Q=Q(x)$ with $Q(x)\ll\sqrt{x}$, we will make a specific choice later.  With our setup matching that of Theorem~\ref{T:LS for Frobenius}, we define 
\[
\Ss(x) = \big\{ \p \in \Sigma_k(x) : \; \p\in S_{\ell} \;\text{ or }\; \rho_{\ell}(\Frob_\p) \subseteq C_\ell  \text{ for all $\ell \in \Lambda(Q)$}\big\}.
\]
Note that $\calS(x) \subseteq \Ss(x)$, so it suffices to bound $|\Ss(x)|$.

Before applying the large sieve, we first calculate some related quantities.  The calculations reduce to counting various elements of $\GL_2(\ZZ/\ell\ZZ)$, see Lemma~\ref{L:misc bounds for GL}.  For any $\ell\in\Lambda(Q)$:
\begin{align*}
|\GG(\ZZ/\ell\ZZ)| &= (\ell-1)|\SL_2(\ZZ/\ell\ZZ)|^n= \ell^n(\ell-1)^{n+1}(\ell+1)^n\\
|C_\ell| &= \sum_{(t_1,\ldots,t_n,d )\in \Omega_\ell } \prod_{i=1}^n |\{A \in \GL_2(\ZZ/\ell\ZZ): \tr(A)=t_i, \, \det(A)=d \}|\\
&\leq \sum_{(t_1,\ldots,t_n,d )\in \Omega_\ell } (\ell(\ell+1))^n 
=  |\Omega_\ell|(\ell(\ell+1))^n \leq c \ell^{2n+1}(\ell+1)^n\\
|C_\ell|/|\GG(\ZZ/\ell\ZZ)| &\leq c (1+ (\ell-1)^{-1})^{n+1} 
\end{align*}
After possibly removing finitely many primes from $\Lambda$, there is a constant $c'<1$ such that $|C_\ell|/|\GG(\ZZ/\ell\ZZ)| \leq c'$ for all $\ell\in \Lambda$.  By Lemma~\ref{L:Gallagher}(iii), $|\GG(\ZZ/\ell\ZZ)^\sharp|\leq (\ell-1)|\SL_2(\ZZ/\ell\ZZ)^\sharp|^n$, and by Lemma~\ref{L:misc bounds for GSp} there is an absolute constant $\kappa\geq 1$ such that $|\GG(\ZZ/\ell\ZZ)^\sharp| \leq \kappa^n \ell^{n+1}\leq (\kappa\ell)^{n+1}$.\\

Define $\Zz(Q)=\{D : D\subseteq \Lambda(Q),\, \prod_{\ell\in D}\kappa\ell \leq Q\}$ 
and
$L(Q) = \sum_{D\in\Zz(Q)} \prod_{\ell\in D} (1-c')/c'$.  
Since $\Lambda$ has positive density, we have
\[
L(Q) \geq \sum_{\ell\in\Lambda, \; \ell \leq Q/\kappa} \frac{1-c'}{c'} \gg \frac{Q}{\log Q}.
\]
For $D\in \Zz(Q)$, define $G_D=\prod_{\ell\in D}\GG(\ZZ/\ell\ZZ)$.
\begin{align*}
&|G_D|=\prod_{\ell\in D}|\GG(\ZZ/\ell\ZZ)| \leq \Bigl(\prod_{\ell\in D} \ell\Bigr)^{3n+1} \leq (Q/\kappa^{ |D|})^{3n+1}\leq Q^{3n+1}\\
&|G_D^\sharp|=\prod_{\ell\in D}|\GG(\ZZ/\ell\ZZ)^\sharp| \leq \Bigl(\prod_{\ell\in D} \kappa \ell \Bigr)^{n+1} \leq Q^{n+1}\\
&\sum_{D\in\Zz(Q)} |G_D^\sharp| |G_D| \leq |\Zz(Q)| Q^{4n+2} \leq Q^{4n+3}
\end{align*}
Applying Theorem~\ref{T:LS for Frobenius}(i) with $r=3n+1$, 
\[ 
L(Q) \gg Q/\log Q \gg \frac{(\log x)^{1/(6r)}}{(\log\log x)^{1+1/(3r)}} 
\]
and hence
\[ 
|\Ss(x)| \ll (x/\log x)/L(Q)\ll \frac{x(\log\log x)^{1+1/(9n+3)}}{(\log x)^{1+1/(18n+6)} }.
\]
Assuming GRH, by Theorem~\ref{T:LS for Frobenius}(ii) 
\[ 
|\Ss(x)| \ll (x/\log x +  Q^{7n+4} x^{1/2}\log x )/(Q/\log Q); 
\]
choosing $Q(x)= (x^{1/2}/(\log x)^2)^{1/(7n+4)}$ gives
\[ 
|\Ss(x)| \ll x^{1-1/(14n+8)} (\log x)^{2/(7n+4)}. 
\]

\subsubsection{Thin sets of Type $1$}
Now assume that $\Omega$ is thin of type $1$; i.e., $\Omega$ is not Zariski dense in $\AA_\QQ^{n+1}$.  By Lemma~\ref{L:reduction of thin sets}, there is a constant $C>0$ such that $|\Omega_\ell|\leq C\ell^{n}$ for all primes $\ell$.  For each $\ell$, define 
$C_\ell = \{ (A_i) \in \GG(\ZZ/\ell\ZZ) : (\tr(A_1),\ldots,\tr(A_n),\det(A_1)) \in \Omega_\ell\}$.
Arguing as before, we find the following bounds.
\begin{align*}
|C_\ell| &= \sum_{(t_1,\ldots,t_n,d )\in \Omega_\ell } \prod_{i=1}^n |\{A \in \GL_2(\ZZ/\ell\ZZ): \tr(A)=t_i, \det(A)=d \}|\\
&\leq \sum_{(t_1,\ldots,t_n,d )\in \Omega_\ell } (\ell(\ell+1))^n 
=  |\Omega_\ell|(\ell(\ell+1))^n \leq C \ell^{2n}(\ell+1)^n
\end{align*}
By possibly increasing the value of $C$, we always have $|C_\ell|/|\GG(\ZZ/\ell\ZZ)| \leq C/\ell.$
Take any prime $\ell\nmid B$, where $B$ is the constant from Theorem~\ref{T:Elliptic curves with large monodromy}.  Let $L_\ell$ be the fixed field of $\ker(\rho_\ell)$ in $\kbar$; there is an isomorphism
\[
\rho_\ell\colon \Gal(L_\ell/k) \overset{\sim}{\to} \GG(\ZZ/\ell\ZZ),
\]
which we will use as an identification.  For all $\p\in \Sigma_k-S$ with $\p\nmid \ell$, $\rho_\ell(\Frob_\p)\subseteq C_\ell$.  Therefore,
\[
|\{\p \in \Sigma_k(x) : (a_\p(E_1),\dots,a_\p(E_n),N(\p))\in\Omega\}|
\leq \pi_{C_\ell}(x,L_\ell/k) + O(1);
\]
see \S\ref{SS:basics} for notation.  It thus suffices to bound $\pi_{C_\ell}(x,L_\ell/k)$, and this can be done using the effective versions of the Chebotarev density theorem given in Appendix~\ref{A:Character sums}.  We first calculate $M(L_\ell/k)$ (see Definition~\ref{D:M definition}).
\[
M(L_\ell/k) = [L_\ell:k] d_k^{1/[k:\QQ]} \prod_{p\in P(L/k)} p \leq \ell^{3n+1} d_k^{1/[k:\QQ]}\cdot \ell \prod_{\p\in S} N(\p)  \ll \ell^{3n+2}
\]
Assuming GRH (and assuming, say, $\ell\leq x$), by Proposition~\ref{P:conditional CDT}(i),
\begin{align*}
\pi_{C_\ell}(x,L_\ell/k) & \leq \frac{|C_\ell|}{|\GG(\ZZ/\ell\ZZ)|} \Li x + O(|C_\ell| x^{1/2} \log x)  \ll \frac{1}{\ell}\frac{x}{\log x} + \ell^{3n} x^{1/2}\log x.
\end{align*}
Choose $\ell \nmid B$ such that 
\[   \big(x^{1/2}/(\log x)^2\big)^{1/(3n+1)} \leq \ell \leq 2\big(x^{1/2}/(\log x)^2\big)^{1/(3n+1)}
\]
(this can be done assuming $x$ is sufficiently large).  With this choice of $\ell$, 
\[
\pi_{C_\ell}(x,L_\ell/k) \ll \frac{x^{1-1/(6n+2)}}{(\log x)^{1 - 2/(3n+1)}}.
\]
Now consider the unconditional case.  By Proposition~\ref{P:unconditional CDT},
\[
\pi_{C_\ell}(x,L_\ell/k) \ll \frac{1}{\ell} \frac{x}{\log x},
\]
assuming that 
\begin{equation}\label{E:lots of logs}
\log x \geq c_2 (\log d_{L_\ell})(\log\log d_{L_\ell})(\log\log\log 6 d_{L_\ell}), 
\end{equation}
where $c_2$ is some absolute constant.  By Lemma~\ref{L:discriminant bound} and the above bound for $M(L_\ell/k)$, 
\begin{equation*} 
\log d_{L_\ell} \leq [L_\ell:\QQ] \log M(L_\ell/k) \ll \ell^{3n+1} \log \ell,
\end{equation*}
and hence
\begin{equation}\label{E:discbound}
(\log d_{L_\ell})(\log\log d_{L_\ell})(\log\log\log 6 d_{L_\ell}) 
\ll \ell^{3n+1} (\log \ell)^2 (\log\log \ell).
\end{equation}
Let $c>0$ be a constant which will be chosen sufficiently small, and suppose we have $\ell \nmid B$ with
\[  c \Bigl(\frac{\log x}{(\log\log x)^2(\log\log\log x)}\Bigr)^{1/(3n+1)}  \leq \ell \leq 2c \Bigl(\frac{\log x}{(\log\log x)^2(\log\log\log x)}\Bigr)^{1/(3n+1)}.
\]
For $c>0$ sufficiently small, the bound (\ref{E:discbound}) shows that (\ref{E:lots of logs}) will hold.  With such an $\ell$,
\[
\pi_{C_\ell}(x,L_\ell/k) \ll \frac{1}{\ell} \frac{x}{\log x}
\ll \frac{x(\log\log x)^{2/(3n+1)}(\log\log\log x)^{1/(3n+1)}}{(\log x)^{1+1/(3n+1)}}.
\]
Such a prime $\ell$ will exist assuming $x$ is sufficiently large.

\begin{remark} \label{R:sieving by one prime with AHC}
If we assume GRH and AHC, then for $\ell\leq x$, Proposition~\ref{P:conditional CDT}(ii) gives
\[
\pi_{C_\ell}(x,L_\ell/k) \ll \frac{1}{\ell}\frac{x}{\log x}+ \ell^{3n/2} x^{1/2}\log x.
\]
Choosing $\ell \approx (x^{1/2}/(\log x)^2)^{2/(3n+2)}$ gives the bound
\[
|\{\p \in \Sigma_k(x) : (a_\p(E_1),\dots,a_\p(E_n),N(\p))\in\Omega\}|
\ll \frac{x^{1-1/(3n+2)}}{(\log x)^{1-4/(3n+2)}}.
\]
\end{remark}

\section{Explicit Chavdarov} \label{S:Chavdarov}
The purpose of this section is to prove Theorem~\ref{T:Effective Chavdarov}.  We keep the notation introduced in \S\ref{SS:Chavdarov intro}.
\subsection{Group theory of $\Sym_n$}
For a positive integer $n$, let $\Sym_{n}$ be the symmetric group on $\{1,2,\dots, n\}$.  A \defi{partition} of $n$ is a sequence $\sigma=(\sigma_1,\dots,\sigma_k)$ of integers such that $n=\sum_{i}\sigma_i$ and $\sigma_1\geq \dots\geq \sigma_k\geq 1$.  The \defi{cycle type} of a permutation $\tau \in \Sym_n$ is the partition $\sigma=(\sigma_1,\dots,\sigma_k)$ of $n$ for which $\tau$ can be written as a product of disjoint cycles of lengths $\sigma_1,\dots,\sigma_k$.  

Let $f(T)\in \ZZ[T]$ be a separable polynomial of degree $n$ with roots $\alpha_1,\dots,\alpha_n$ in $\Qbar$.  The numbering of the roots induces an injective homomorphism
$\Gal(f(T)) \hookrightarrow \Sym_{n}.$
This homomorphism, up to an inner automorphism of $\Sym_n$, is independent of the choice of numbering.
\begin{defn}
Let $f(T)\in\ZZ[T]$ be a polynomial of degree $n$ and let $\sigma=(\sigma_1,\dots,\sigma_k)$ be a partition of $n$.   We say that $\sigma$ is a \defi{cycle type} of $f$ if $f$ is separable and the image of $\Gal(f(T))\hookrightarrow \Sym_{n}$ contains a permutation with cycle type $\sigma$.  An equivalent condition (by the Chebotarev density theorem) is that there exists a prime $\ell$ such that $f(T) \bmod{\ell} \in \FF_\ell[T]$ factors into distinct irreducibles of degrees $\sigma_1,\dots,\sigma_k$.
\end{defn}

\begin{lemma} \label{L:full Galois}
Let $f(T)\in\ZZ[T]$ be a polynomial of degree $n$.  Suppose that $\sigma$ is a cycle type of $f(T)$ for each partition $\sigma$ of $n$.  Then $\Gal(f(T))\cong\Sym_n$.
\end{lemma}
\begin{proof}
The cycle type of a permutation in $\Sym_n$ induces a bijection between partitions of $n$ and conjugacy classes of $\Sym_n$.  Our assumption implies that the image of $\Gal(f(T))\hookrightarrow \Sym_n$ meets every conjugacy class of $\Sym_n$.  A classical lemma of Jordan, says that for each proper subgroup $H$ of a finite group $G$, there is a conjugacy class $C\in G^\sharp$ such that $H\cap C=\emptyset$.  Therefore, $\Gal(f(T))\cong \Sym_n$. 
\end{proof}
\subsection{Group theory of $W_{2g}$}
Fix an integer $g\geq 1$.  Recall that $W_{2g}$ is the subgroup of $\Sym_{2g}$ which induces an permutation on the set of pairs $\big\{\{1,2\},\{3,4\},\dots,\{2g-1,2g\} \big\}$.  The action of $W_{2g}$ on these $g$ pairs gives an exact sequence
\begin{equation} \label{E:Weyl group sequence}
1\to H \to W_{2g} \overset{\phi}{\to} \Sym_g \to 1.
\end{equation}
The group $H$ is generated by the transpositions $(1,2),\dots, (2g-1,2g)$, and hence is isomorphic to $(\ZZ/2\ZZ)^{g}$.  In particular, $|W_{2g}|=2^g g!$.

\begin{lemma} \label{L:W group theory}
Let $G$ be a subgroup of $W_{2g}$.  If $G$ contains a transposition and $\phi(G)=\Sym_g$, then $G=W_{2g}$.
\end{lemma}
\begin{proof}
From the assumption $\phi(G)=\Sym_g$ and (\ref{E:Weyl group sequence}), it suffices to show that $H\subseteq G$.  In particular, it suffices to show that $G$ contains every transposition of the form $(2i-1,2i)$.

By assumption, $G$ contains a transposition.  This transposition must be an element of $H$, and we may assume that it is $(1,2)$.  Since $\phi(G)=\Sym_g$, there exists a $\tau\in G$ which switches the pairs $\{1,2\}$ and $\{2i-1,2i\}$, and leaves the other pairs fixed.  The permutation $\tau (1,2) \tau^{-1}$ is thus $(2i-1,2i)$.
\end{proof}

\subsection{Abelian varieties over finite fields}

\begin{defn}
Let $A$ be an abelian variety of dimension $g\geq 1$ over a finite field $\FF_q$.  Let $Q_A(T)\in\ZZ[T]$ be the unique polynomial such that $P_A(T)=T^g Q_A(T+q/T)$.  (The existence of $Q_A(T)$ is a direct consequence of the functional equation $P_A(q/T)/(q/T)^g = P_A(T)/T^g$).
\end{defn}

\begin{lemma} \label{L:W criterion}
Let $A$ be an abelian variety of dimension $g$ over $\FF_q$.  Suppose that $P_A(T)$ has cycle types $(2g)$ and $(2,1,1,\dots,1)$, and $Q_A(T)$ has cycle type $\sigma $ for each partition $\sigma$ of $g$.  Then $\Gal(P_A(T))\cong W_{2g}$.
\end{lemma}
\begin{proof}
The polynomial $P_A(T)$ is irreducible since it has cycle type $(2g)$.  Let $\pi_1,\dots,\pi_{2g}$ be the roots of $P_A(T)$ in $\Qbar$, they are non-rational and distinct since $P_A(T)$ is irreducible.  We may assume that the $\pi_i$ are numbered such that the product of any of the pairs $\{\pi_{1},\pi_2\},\dots,$ $\{\pi_{2g-1},$ $\pi_{2g}\}$ is $q$ (since $P_A(T)$ is irreducible of degree $2g$, $\pm\sqrt{q}$ can be roots only when $g=1$, in which case the lemma is trivial).  The numbering of the $\pi_i$ induces an injective homomorphism $\Gal(P_A(T)) \hookrightarrow W_{2g}$; let $G$ be the image of this map.  Since $P_A(T)$ has cycle type $(2,1,1,\dots,1)$, the group $G\subseteq \Sym_{2g}$ contains a transposition.

The polynomial $Q_A(T)$ is monic of degree $g$.  Since the value of $T+q/T$ at any element of a pair $\{\pi_{2i-1},\pi_{2i}\}$ is the same, we find that roots of $Q_A(T)$ correspond with our $g$ pairs of roots of $P_A(T)$.  By Lemma~\ref{L:full Galois}, $\Gal(Q_A(T))\cong \Sym_g$.  Thus $\phi(G)=\Sym_g$, where $\phi$ is the map from (\ref{E:Weyl group sequence}).  By Lemma~\ref{L:W group theory}, $\Gal(P_A(T))\cong G=W_{2g}$ as desired.
\end{proof}

\begin{defn}  
Let $P(T)\in \QQ[T]$ be a monic polynomial of degree $n$, and let $\alpha_1,\dots,\alpha_n\in \Qbar$ be the roots of $P(T)$.  For each integer $m>0$, define 
$P^{(m)}(T):= \prod_i(T-\alpha_i^m)$ which is a well-defined element of $\QQ[T]$.
\end{defn}
Let $A$ be an abelian variety over the finite field $\FF_q$ with $q$ elements.  Then for all $m\geq 1$,
$P_A^{(m)}(T) = P_{A \times \FF_{q^m}}(T).$
The next lemma gives a useful criterion to test whether $\Gal(P_A^{(m)}(T)) \cong \Gal(P_A(T))$ for all $m\geq 1$. 
\begin{lemma} \label{L:geometric criterion}
Let $P(T)\in\QQ[T]$ be an irreducible polynomial of degree $n$ with Galois group $G$.  There is an integer $s=s(n)>0$, depending only on the degree of $P(T)$, such that if the polynomial $P^{(s)}(T)$ is separable, then the polynomial $P^{(m)}(T)\in\QQ[T]$ is irreducible and has Galois group $G$ for all integers $m\geq 1$.
\end{lemma}
\begin{proof}
This follows from \cite{Chavdarov}*{Lemma 5.3}.
\end{proof}

\subsection{Proof of Theorem \ref{T:Effective Chavdarov}}
Fix a positive function $Q=Q(x)$ with $Q\ll \sqrt{x}$ which will be specifically chosen later. Define the set $\Lambda=\{ \ell : \ell\nmid B \}$, where $B$ is the constant from Theorem~\ref{T:Abelian variety with full GSp}.  Thus the representations $\{\rho_{A,\ell}\}_{\ell\in\Lambda}$ are independent.  Let $\Lambda(Q)$ be the set of elements in $\Lambda$ which are at most $Q$.  
\begin{itemize}
\item 
Let $\calS_1(x)$ be the set of $\p\in \Sigma_k(x)-S_A$ such that $P_{A_\p}(T)$ does not have cycle type $(2g)$.
\item
Let $\calS_2(x)$ be the set of $\p\in \Sigma_k(x)-S_A$ such that $P_{A_\p}(T)$ does not have cycle type $(2,1,1\dots,1)$.
\item
Let $\calS_3(x)$ be the set of $\p\in \Sigma_k(x)-S_A$ such that $P_{A_\p}^{(s)}(T)$ is not separable, where $s=s(2g)$ is the integer from Lemma \ref{L:geometric criterion}.
\item 
For each partition $\sigma$ of $g$, let $\calS_\sigma(x)$ be the set of $\p\in \Sigma_k(x)-S_A$ such that $Q_{A_\p}(T)$ does not have cycle type $\sigma$.
\end{itemize}

\begin{lemma} \label{L:broken up bounds}
There is an inclusion $\Pi_A(x) \subseteq \calS_1(x) \cup \calS_2(x) \cup \calS_3(x) \cup \bigcup_{\sigma} \calS_\sigma(x)$, and hence
\[ 
|\Pi_A(x)| \leq |\calS_1(x)| + |\calS_2(x)| + |\calS_3(x)| + \sum_{\sigma} |\calS_\sigma(x)|.
\]
\end{lemma}
\begin{proof}
Take any $\p\in \Sigma_k(x)-S_A$ with $\p\not \in \calS_1(x) \cup \calS_2(x) \cup \calS_3(x) \cup \bigcup_{\sigma} \calS_\sigma(x)$.  
Since $\p\not\in \calS_1(x) \cup \calS_2(x)$, $P_{A_\p}(T)$ has cycle types $(2g)$ and $(2,1,1,\dots,1)$.  For each partition $\sigma$ of $g$, $\p\not\in\calS_\sigma(x)$ imples that $Q_{A_\p}(T)$ has cycle type $\sigma$.  By Lemma~\ref{L:W criterion}, $\Gal(P_{A_\p}(T))\cong W_{2g}$.
Since $\p\not\in \calS_3(x)$, we find by Lemma~\ref{L:geometric criterion} that 
\[ 
\Gal(P_{A_\p}^{(m)}(T))\cong \Gal(P_{A_\p}(T))\cong W_{2g}
\]
for all $m\geq 1$.  Therefore, $\p\notin \Pi_A(x)$.
\end{proof}
For each prime $\ell$, define the following sets:
\begin{itemize}
\item
Let $C_\ell^1$ be the set of $B \in \GSp_{2g}(\FF_\ell)$ such that $\det(TI-B)\in \FF_\ell[T]$ is reducible.
\item
Let $C_\ell^2$ be the set $B \in \GSp_{2g}(\FF_\ell)$ such that $\det(TI-B)\in \FF_\ell[T]$ is not the product of an irreducible quadratic and $2g-2$ distinct linear terms.
\item
Let $s=s(2g)$ be the integer of Lemma \ref{L:geometric criterion}.  Let $C_\ell^3$ be the set of $B \in \GSp_{2g}(\FF_\ell)$ such that $P^{(s)}(T) \in \FF_\ell[T]$ is not separable, where $P(T)=\det(TI-B)$.
\item
For each partition $\sigma=(\sigma_1,\dots,\sigma_k)$ of $g$, let
$C_\ell^\sigma$ be the set of $B \in \GSp_{2g}(\FF_\ell)$ such that $Q(T)$ does not factor into distinct irreducible polynomials of degree $\sigma_1,\dots,\sigma_k$, where $Q(T)\in \FF_\ell[T]$ is the unique polynomial such that $\det(TI-B)=T^g Q(T+\mult(B)/T)$.
\end{itemize}

\begin{lemma} \label{L:Chavdarov bounds}
There are constants $B'>0$ and $0<\delta<1$, depending only on $g$,
such that for all primes $\ell\geq B'$,
\begin{equation} \label{E:max over C's}
\max_{i=1,2,3}\frac{|C_\ell^i|}{|\GSp_{2g}(\FF_\ell)|} \leq \delta \quad \text{ and }\quad
\max_{\sigma \text{ partition of }g} \frac{|C_\ell^\sigma|}{|\GSp_{2g}(\FF_\ell)| }\leq \delta.
\end{equation}
\end{lemma}
\begin{proof}
These bounds follow from the computations done in \cite{Chavdarov} (in particular, see Corollary 3.6, Lemma 5.7, Lemma 5.4, and Lemma 5.9).  Chavdarov's bounds are done for the $\Sp_{2g}(\FF_\ell)$ cosets of $\GSp_{2g}(\FF_\ell)$, our lemma follows by combining these bounds.  Also note that the formulation of some of these results looks slightly different in \cite{Chavdarov} because the characteristic polynomials there are the reverse of ours.
\end{proof}

We have reduced to the case of bounding the following cardinalities separately: $|\calS_1(x)|,|\calS_2(x)|,|\calS_3(x)|$, and $|\calS_\sigma(x)|$ for each partition $\sigma$ of $g$.  For purely notational reasons we only bound $|\calS_1(x)|$; the arguments in the other cases are identical.  For any $\ell\in\Lambda(Q)$ and $\p\in \calS_1(x)$, either $\p|\ell$ or $\rho_{A,\ell}(\Frob_\p) \subseteq C_\ell^1$.  So $\calS_1(x) \subseteq \Ss_1(x)$, where
\[
\Ss_1(x) := \big\{ \p \in \Sigma_k(x)-S_A :  \p |\ell \; \text{ or } \; \rho_{A,\ell}(\Frob_\p) \subseteq C_\ell \text{ for all $\ell\in \Lambda(Q)$}\big\}.
\]
We will now use the large sieve as in Theorem~\ref{T:LS for Frobenius} to bound $|\Ss_1(x)|$.

By possibly increasing $B$, we may assume that (\ref{E:max over C's}) holds for all $\ell\nmid B$.  By Lemma~\ref{L:misc bounds for GSp}, there is a constant $\kappa\geq 1$ such that $|\GSp_{2g}(\FF_\ell)^\sharp| \leq (\kappa\ell)^{g+1}$.
Define 
\[
\Zz(Q)=\{D: D\subseteq \Lambda(Q),\, {\prod}_{\ell\in D}\kappa\ell \leq Q\} \text{\quad and\quad}
L(Q)=\sum_{D \in \Zz(Q)} \prod_{\ell\in D} \frac{1-\delta}{\delta},
\]
where $\delta$ is the constant from Lemma~\ref{L:Chavdarov bounds}.  Note that
\[
L(Q) \geq  \sum_{\ell \in \Lambda, \; \ell \leq Q/\kappa}  \frac{1-\delta}{\delta} \gg \frac{Q}{\log Q}.
\]
For $D\in \Zz(Q)$, define $G_D=\prod_{\ell\in D} \GSp_{2g}(\FF_\ell)$.  We shall use Lemma~\ref{L:misc bounds for GSp} in the following bounds.
\[
|G_D|=\prod_{\ell\in D}|\GSp_{2g}(\FF_\ell)| \leq \Bigl(\prod_{\ell\in D} \ell\Bigr)^{2g^2+g+1} \leq (Q/\kappa^{|D|})^{2g^2+g+1} \leq Q^{2g^2+g+1}
\]
\[
|G_D^\sharp|=\prod_{\ell\in D}|\GSp_{2g}(\FF_\ell)^\sharp| \leq \Bigl(\prod_{\ell\in D} \kappa\ell \Bigr)^{g+1}  \leq Q^{g+1} 
\]
\[
\sum_{D\in\Zz(Q)} |G_D^\sharp| |G_D| \ll |\Zz(Q)| Q^{2g^2+2g+2} \ll Q^{2g^2+2g+3}
\]
Applying Theorem~\ref{T:LS for Frobenius}(i) with $r=2g^2+g+1$ gives 
\[ L(Q) \gg Q/\log Q \gg \frac{(\log x)^{1/(6r)}}{(\log\log x)^{1+1/(3r)}} \]
and hence
\[ 
|\calS_1(x)|\leq |\Ss_1(x)| \ll (x/\log x)/L(Q)\ll \frac{x(\log\log x)^{1+1/(6g^2+3g+3)}}{(\log x)^{1+1/(12g^2+6g+6)} }.
\]
Assuming GRH, by Theorem~\ref{T:LS for Frobenius}(ii)
\[ 
|\Ss_1(x)| \ll (x/\log x +  Q^{4g^2+3g+4} x^{1/2}\log x )/(Q/\log Q); \]
choosing $Q(x)= (x^{1/2}/(\log x)^2)^{1/(4g^2+3g+4)}$ gives
\[ 
|\calS_1(x)|\leq |\Ss_1(x)| \ll x^{1-1/(8g^2+6g+8)} (\log x)^{2/(4g^2+3g+4)}. 
\]

Identical bounds will hold for all the $|\calS_2(x)|,$ $|\calS_3(x)|$ and $|\calS_\sigma(x)|$.  So by Lemma~\ref{L:broken up bounds}, we have
\begin{align*}
|\Pi_A(x)| &  \ll
\begin{cases}
\displaystyle {x(\log\log x)^{1+1/(6g^2+3g+3)}}/{(\log x)^{1+1/(12g^2+6g+6)} } & \\[1.0em]
\displaystyle  x^{1-1/(8g^2+6g+8)} (\log x)^{2/(4g^2+3g+4)} &  \text{assuming GRH.} 
\end{cases}
\end{align*}

\section{The Lang-Trotter Conjecture} \label{S:Lang-Trotter}

The purpose of this section is to give an application of our large sieve to a problem for which there is a priori results that can be used as a benchmark to measure how effective our sieve is.

Fix an $E$ be an elliptic curve without complex multiplication that is defined over $\QQ$.  For an integer $t$, define
\[
\Pi_{E,t}(x):= |\{p\leq x:  a_p(E)=t \}|.
\]
With notation as above, we have the following well-known conjecture of Lang and Trotter \cite{Lang-Trotter}.
\begin{conj}[Lang-Trotter]
There is an explicit constant $C_{E,t}\geq 0$ such that as $x\to\infty$,
\[
\Pi_{E,t}(x) \sim C_{E,t} \frac{x^{1/2}}{\log x}.
\]
If $C_{E,t}=0$, then this defined to mean that there are only finitely many primes $p$ with $a_p(E)=t$.
\end{conj}
Theorem~\ref{T:LS for elliptic curves} (with $n=1$, $\Omega=\{t\}\times \ZZ$) gives immediate upper bounds on $\Pi_{E,t}(x)$.  If we assume both GRH and AHC, then Remark~\ref{R:sieving by one prime with AHC} (which does not use the large sieve) gives the bound
\begin{equation} \label{E:MMS bound}
\Pi_{E, t}(x) \ll x^{4/5} / (\log x)^{1/5}.
\end{equation}
Murty, Murty, and Saradha \cite{MMS} have proven (\ref{E:MMS bound}) assuming GRH (but not AHC!).  
\begin{thm} \cite{MMS}
Let $E$ be a non-CM elliptic curve over $\QQ$ and let $t$ be an integer.  Assuming GRH, we have $\Pi_{E, t}(x) \ll x^{4/5} / (\log x)^{1/5}$.
\end{thm}

Their proof reduces the bound to an application of an effective version of the Chebotarev density theorem to abelian extensions (where AHC is known to hold!).  The result in \cite{MMS} is actually stated for modular forms but the elliptic curve proof is identical.\\

The goal of \S\ref{S:Lang-Trotter} is simply to prove, assuming GRH and AHC, the bound (\ref{E:MMS bound}) by using the large sieve of Theorem~\ref{T:LS for Frobenius}.  \\

Before continuing, it is necessary to explain what this is meant to demonstrate.  Recall that the large sieve inequality used in Theorem~\ref{T:LS for Frobenius} comes from the easy bound of Proposition~\ref{P:LSI} and many character sum estimates from Appendix~\ref{A:Character sums}.   That we can recover known bounds, shows that these estimates for the large sieve inequality are not so bad.

One would hope that ``on average'' the error terms in these character sum estimates are small, and thus a stronger large sieve inequality should be true\footnote{This leads to other natural questions; for example, what is the elliptic curve analogue of the Bombieri-Vinogradov theorem?}.   This example shows that any interesting improvement in the large sieve inequality, over the somewhat naive approach used in this paper, would have important arithmetic consequences.  
\\

For simplicity, we will assume that $t\neq 0$.  One can prove stronger bounds in the $t=0$ case by using the corresponding Galois representations $\calG_\QQ \to \PGL_2(\ZZ/\ell\ZZ)$.  In fact, Elkies \cite{Elkies92} has shown \emph{unconditionally} that $\Pi_{E,0}(x)\ll x^{3/4}$ .  For $t\neq 0$, it is still unknown (unconditionally) whether $\Pi_{E,t}(x) \ll x^{1-\delta}$ for some $\delta>0$.

\subsection{Sieve setup} \label{S: LT sieve setup}

By Theorem~\ref{T:Abelian variety with full GSp} (with $g=1$), there is a positive integer $B$ such that 
\[
\rho_{E,m}(\calG_\QQ) = \GL_2(\ZZ/m\ZZ)
\]
for all integers $m$ relatively prime to $B$.  We may assume that $B$ is divisible by the prime factors of $2t$.   Fix a positive function $Q=Q(x)$, to be chosen later, such that $Q(x) \ll \sqrt{x}$.  
Define the sets 
\[
\Lambda(Q)=\{\ell: \ell < Q, \, \ell \nmid B\} \text{\quad  and \quad } \calS(x) = \{ p\in \Sigma_\QQ(x)  :  a_p(E)=t \}.
\]
For each $\ell\in\Lambda(Q)$ and $p \in \calS(x)$, either $p \in S_E \cup \{\ell\}$ or
\[
\rho_{E,\ell}(\Frob_p)\subseteq C_\ell := \{A\in\GL_2(\ZZ/\ell\ZZ) :  \tr(A) \equiv t \bmod{\ell}\}.
\]
With our setup matching that of Theorem~\ref{T:LS for Frobenius}, we define 
\[
\Ss(x) := \big\{ p \in \Sigma_\QQ(x)  :  p \in S_E \cup\{\ell\}  \; \text{ or }\; \rho_{E,\ell}(\Frob_p) \subseteq C_\ell  \text{ for all $\ell\in \Lambda(Q)$}\big\}.
\]
Note that $\calS(x) \subseteq \Ss(x)$, so it suffices to find upper bounds for $|\Ss(x)|$.

Define the set $\Zz(Q)=\{D :D \subseteq \Lambda(Q),\, \prod_{\ell\in D}(\ell+1) \leq Q\}$, 
and
\[
L(Q) = \sum_{D\in\Zz(Q)} \prod_{\ell\in D} \frac{1-|C_\ell|/|\GL_2(\ZZ/\ell\ZZ)|}{|C_\ell|/|\GL_2(\ZZ/\ell\ZZ)|}.
\]
For $D\in \Zz(Q)$, we define $G_D=\prod_{\ell\in D}\GL_2(\ZZ/\ell\ZZ)$.  Fix an element $\ell\in\Lambda(Q)$.  Using the Cauchy-Schwarz inequality and Lemma~\ref{L:misc bounds for GL}, we have
\begin{align*} 
 \sum_{ \chi \in \Irr(\GL(\ZZ/\ell\ZZ))} \chi(1)
& \leq \Bigl( \sum_{ \chi \in \Irr(\GL(\ZZ/\ell\ZZ))} \chi(1)^2 \Bigr)^{1/2} |\GL(\ZZ/\ell\ZZ)^\#|^{1/2}\\ 
&= {|\GL(\ZZ/\ell\ZZ)|}^{1/2} {|\GL(\ZZ/\ell\ZZ)^\#|}^{1/2} \leq \sqrt{ \ell^4 } \sqrt{ \ell^2} = \ell^3.
\end{align*}
Therefore,
\[
\sum_{D\in\Zz(Q)} \sum_{\chi\in\Irr(G_D)} \chi(1) \leq \sum_{D\in\Zz(Q)} \prod_{\ell\in D} \Bigl( \sum_{\chi \in \Irr(\GL_2(\ZZ/\ell\ZZ))} \chi(1) \Bigr) \leq \sum_{D\in\Zz(Q)}\Bigl(\prod_{\ell\in D} \ell\Bigr)^3 \leq |\Zz(Q)|Q^3 \leq Q^4.
\]
For each $\ell \nmid B$, from the description of the characters of $\GL_2(\ZZ/\ell\ZZ)$ in \cite{LangAlgebra}*{XVIII, \S12},  
\[
\max_{\chi\in \Irr(\GL_2(\ZZ/\ell\ZZ))} \chi(1) = \ell+1,
\]
and thus
\[
\max_{D \in \Zz(Q), \, \chi\in \Irr(G_D)} \chi(1) = \max_{D\in\Zz(Q)} {\prod}_{\ell\in D} (\ell+1) \leq Q.
\]
By Theorem~\ref{T:LS for Frobenius}, assuming AHC and GRH, we have
\begin{equation} \label{E:LS bound for LT1}
\Pi_{E,t}(x) = |\calS(x)| \leq  |\Ss(x)| \leq (\Li x + O(Q^5x^{1/2} \log x)) L(Q)^{-1}.
\end{equation}

\subsection{Asymptotics of $L(Q)$}
In this section, we will prove an asymptotic lower bound for $L(Q)$.  By Lemma~\ref{L:misc bounds for GL},  for any $\ell\in\Lambda(Q)$
\[
|C_\ell| = \ell(\ell^2-\ell-1) \text{\quad and \quad} \frac{|C_\ell|}{|\GL_2(\ZZ/\ell\ZZ)|} = \frac{\ell^2-\ell-1}{(\ell-1)^2(\ell+1)},
\]
so
\[
L(Q) = \sum_{D\in\Zz(Q)} \prod_{\ell\in D} \Bigl(\ell-1 + \frac{1}{\ell^2 - \ell - 1}\Bigr) \geq \sum_{D\in\Zz(Q)} \prod_{\ell\in D} (\ell-1).
\]
Let $\mu$ be the M\"obius function, and define the arithmetic functions 
\[
\varphi(n)=n\prod_{\ell|n}(1-1/\ell) \text{ \quad and \quad } \psi(n)=n\prod_{\ell|n}(1+1/\ell).
\]
Thus
\begin{equation}\label{E:LT lower bound on L}
L(Q) = \sum_{D\in \Zz(Q)}  \prod_{\ell \in D} (\ell-1)\geq \sum_{\substack{ \psi(d)\leq Q\\ \gcd(d,\prod_{\ell \leq B} \ell)=1}} \mu^2(d) \varphi(d) \gg \sum_{ \psi(d)\leq Q} \mu^2(d) \varphi(d).
\end{equation}
To find a lower bound for this expression, we will apply the following result on the distribution of the values $\varphi(n)/n$.
\begin{lemma}[\cite{Kac}*{Chapter~4, \S2}] \label{L:cdf}
Define a function $F:\RR \to [0,1]$ by  
\[
F(z):= \lim_{N\to\infty}  \frac{|\{ n \leq N: \varphi(n)/n < z \}|}{N}.
\]
The map $F$ is well-defined (i.e.,~the limits exist) and is continuous.
\end{lemma}

\begin{lemma} \label{L:LT lower bound on L}
$\sum_{\psi(d) \leq Q} \mu^2(d) \varphi(d) \gg Q^2.$
\end{lemma}
\begin{proof}
For any positive integer $d$, 
\[
 d^2\geq \psi(d)\varphi(d) = d^2\prod_{\ell|d} (1-\frac{1}{\ell^2})\geq d^2 \zeta(2)^{-1}= \frac{6}{\pi^2} d^2,
 \]
 and so
\begin{equation} \label{E:psi1} 
\sum_{\psi(d) \leq Q} \mu^2(d) \varphi(d) \geq  \frac{6}{\pi^2} \sum_{\psi(d) \leq Q} \mu^2(d) \frac{d^2}{\psi(d)} \geq   \frac{6}{\pi^2} Q^{-1} \sum_{\psi(d) \leq Q} \mu^2(d) d^2.
\end{equation}
Fix a constant $c$ with $0<c<1$, we then have the following easy inequalities
\begin{align*}
\sum_{\psi(d) \leq Q} \mu^2(d) \varphi(d)  
&\geq \frac{6}{\pi^2}  Q^{-1} \sum_{\substack{cQ/2 \leq d \leq cQ\\ \psi(d)/d \leq  c^{-1}}} \mu^2(d) d^2 
\geq \frac{3c^2}{2\pi^2}  Q  \sum_{\substack{cQ/2 \leq d \leq cQ \\ \psi(d)/d \leq  c^{-1}}} \mu^2(d).
\end{align*}
Since $\varphi(d)/d \geq c $ implies $\psi(d)/d\leq c^{-1}$, we have
\begin{equation} \label{E:psi3} 
\sum_{\psi(d) \leq Q} \mu^2(d) \varphi(d) 
\geq \frac{3c^2}{2\pi^2}  Q  \sum_{cQ/2 \leq d \leq cQ, \, c\leq  \varphi(d)/d } \mu^2(d) \gg c^2 Q |A_Q \cap B_Q|,
\end{equation}
where 
\[
A_Q= \left\{ d \in [cQ/2,cQ] : d \text{ squarefree} \right\} \text{\quad and \quad}
B_Q= \left\{d \in [cQ/2,cQ]:  c \leq \varphi(d)/d  \right\}.
\]
By Lemma \ref{L:cdf}
\[
|B_Q|  = (1-F(c)) {cQ}/{2} +o(Q),
\]
and it is well known that $|A_Q| = (6/\pi^2) cQ/2 +o(Q)$.   
\begin{align*} 
|A_Q \cap B_Q| &= |A_Q| + |B_Q| - |A_Q\cup B_Q| \\
&\geq  |A_Q| + |B_Q| - (cQ/2+1) \\
&=  \frac{6}{\pi^2} \frac{cQ}{2} + (1-F(c)) \frac{cQ}{2} - \frac{cQ}{2} +o(Q)
=  \Big(\frac{6}{\pi^2}- F(c)\Big) \frac{cQ}{2} +o(Q)
\end{align*}
Now choose our constant $c$ such that $F(c) < 6/\pi^2$ (this can be done since $F$ is continuous, and $F(0)=0$, $F(1)=1$).  Equation (\ref{E:psi3}) becomes $\sum_{\psi(d) \leq Q} \mu^2(d) \varphi(d) \gg Q^2$.
\end{proof}
Combining (\ref{E:LT lower bound on L}) and Lemma \ref{L:LT lower bound on L} proves the following:
\begin{equation} \label{E:L estimate} 
L(Q) \gg Q^2. 
\end{equation}
\subsection{Final bound}
Using (\ref{E:LS bound for LT1}) and (\ref{E:L estimate}), we have
\[
\Pi_{E,t}(x) \ll (x/\log x + Q^5x^{1/2} \log x)Q^{-2}.
\]
Setting $Q(x)= x^{1/10} / (\log x)^{2/5}$, we deduce that
\[
\Pi_{E,t}(x) \ll {x^{4/5}}/{(\log x)^{1/5}}
\]
assuming GRH and AHC.

\appendix
\section{Character sums and the Chebotarev density theorem} \label{A:Character sums}
\subsection{Notation} \label{SS:basics}
Let $L/k$ be a Galois extension of number fields with Galois group $G$ and let $C$ be a subset of $G$ stable under conjugation.  Define
\[
\pi_C(x,L/k) := |\{ \p\in \Sigma_k(x) : \p \text{ unramified in $L$, } \Fr_\p \subseteq C \}|.
\] 
The \emph{Chebotarev density theorem} says that 
\[
\pi_C(x,L/k) \sim \frac{|C|}{|G|}\Li x
\]
as $x\to \infty$.  An effective version would give an explicit bound for $\pi_C(x,L/k) - |C|/|G| \Li x.$\\

The extension $L/k$ is said to satisfy \defi{Artin's Holomorphy Conjecture} (AHC) if for each $\chi\in\Irr(G)-\{1\}$, the Artin $L$-series $L(s,\chi)$ has analytic continuation to the whole complex plane.

The \defi{Generalized Riemann Hypothesis} (GRH) asserts that for any number field $L$, the Dedekind zeta function $\zeta_L(s)$ has no zeros with real part $>1/2$. 

\begin{defn} \label{D:M definition}
Let $L/k$ be an extension of number fields.  Define
\[ 
M(L/k) = [L:k] d_k^{1/[k:\QQ]} \prod_{p \in P(L/k)} p
\]
where $d_k$ is the absolute discriminant of $k$ and $P(L/k)$ is the set of rational primes $p$ for which there exist a prime $\p \in \Sigma_k$ such that $\p|p$ and $\p$ is ramified in $L$.
\end{defn}

\begin{lemma} \label{L:discriminant bound}
Let $L/k$ be a Galois extension of number fields.  Then $\log d_L \leq [L:\QQ] \log M(L/k)$.
\end{lemma}
\begin{proof}
This follows by combining equations (3) and (6) of \cite{SerreCheb}.
\end{proof}

\subsection{Conditional versions}
\begin{prop}  \label{P:conditional CDT} 
Let $L/k$ be a Galois extension of number fields with Galois group $G$ and let $C$ be a subset of $G$ stable under conjugation.
\begin{romanenum}
\item Assume GRH.  Then 
\[
\pi_C(x,L/k) = \frac{|C|}{|G|} \Li x + O\Bigl( |C|x^{1/2}[k:\QQ] \log\big( M(L/k) x \big)  \Bigr).\]
\item Assume GRH and assume AHC for the extension $L/k$. Then 
\[
\pi_C(x,L/k)= \frac{|C|}{|G|} \Li x + O\Bigl( |C|^{1/2}x^{1/2}[k:\QQ] \log\big( M(L/k) x \big)  \Bigr).
\]
\end{romanenum}
In both cases, the implicit constants are absolute.
\end{prop}
\begin{proof}
Part (i) is equation $(20_R)$ of \cite{SerreCheb}.  Part (ii) is a consequence of \cite{MMS}*{Proposition 3.12}.
\end{proof}

\begin{prop}  \label{P:conditional character sum} 
Let $L/k$ be a Galois extension of number fields with Galois group $G$ and let $\chi$ be a character of $G$.  
\begin{romanenum}
\item Assume GRH.  Then
\[
\sum_{\substack{\p\in\Sigma_k(x) \\ \p \text{ unramified in }L}} \chi(\Frob_\p) = (\chi,1) \Li x + O\Bigl( \Bigl( {\sum}_{g\in G} |\chi(g)| \Bigr) [k:\QQ] x^{1/2} \log\big(M(L/k)x \big) \Bigr).
\]
\item  Assume GRH and assume AHC for the extension $L/k$.  Then
\[
\sum_{\substack{\p\in\Sigma_k(x) \\ \p \text{ unramified in }L}} \chi(\Frob_\p) = (\chi,1) \Li x + O\Bigl(  \chi(1) [k:\QQ] x^{1/2} \log\big(M(L/k)x \big) \Bigr).
\]
\end{romanenum}
In both cases, the implicit constants are absolute.
\end{prop}
\begin{proof} Part (i) follows from equation ($33_R$) of \cite{SerreCheb} and Lemma~\ref{L:discriminant bound}.
We now consider part (ii).  By additivity it suffices to prove the proposition for an irreducible $\chi$.  Let $\F_\chi$ be the Artin conductor of $\chi$ and define $A_\chi=d_k^{\chi(1)} N_{k/\QQ}(\F_\chi).$  See \cite{MMS}*{Proposition 3.5} for a sketch that
\begin{align*} 
\sum_{\substack{\p\in\Sigma_k(x) \\ \p \text{ unramified in }L}} \chi(\Frob_\p) = (\chi,1)  \Li x &+ O\Bigl(x^{1/2}(\log A_\chi + \chi(1)[k:\QQ]\log x) \Bigr)\\ &+ O\Bigl( \chi(1)[k:\QQ]\log\Bigl([L:k] d_k^{1/[k:\QQ]}{\prod}_{p \in P(L/k)} p\Bigr)\Bigr) .
\end{align*}
By Proposition 2.5 of \cite{MMS},
$\log(N_{k/\QQ}(\F_\chi)) \leq 2\chi(1)[k:\QQ] \log\Bigl([L:k] {\prod}_{p \in P(L/k)} p\Bigr)$
and hence
\begin{align*}
\log A_\chi  
& \leq  2\chi(1)[k:\QQ] \log\Bigl([L:k] {\prod}_{p \in P(L/k)} p\Bigr) + \chi(1) \log d_k \\
&\leq  2\chi(1)[k:\QQ]\log\Bigl([L:k]d_k^{1/[k:\QQ]} {\prod}_{p \in P(L/k)} p\Bigr).
\end{align*}
Combining everything we obtain
\[
\sum_{\substack{\p\in\Sigma_k(x)\\ \p \text{ unramified in }L}}\chi(\Frob_\p) = (\chi,1)  \Li x + O\Bigl(x^{1/2}\chi(1)[k:\QQ] \log\Bigl([L:k]d_k^{1/[k:\QQ]} x {\prod}_{p \in P(L/k)} p \Bigr) \Bigr).   \qedhere
\]
\end{proof}

\subsection{Exceptional zeros}
\begin{prop} \label{P:exceptional zeros}
\begin{romanenum} 
\item
Let $L\neq \QQ$ be a number field.  Then $\zeta_L(s)$ has at most one real zero in the interval $1-(4\log d_L)^{-1} \leq \sigma <1$.  Such a zero of $\zeta_L(s)$, if it exists, is simple.  
\item 
Let $L/k$ be a Galois extension of number fields and suppose that $\beta\geq 1/2$ is a real simple zero of $\zeta_L(s)$.  Then there is a field $F$ with $k\subseteq F\subseteq L$ such that $[F:k]\leq 2$ and $\zeta_F(\beta)=0$.
\item 
Let $F\neq \QQ$ be a number field and suppose $\beta$ is a real zero of $\zeta_F(s)$. Then
\[
1-\beta \gg \min\big\{ ([F:\QQ]!\log d_F)^{-1}, d_F^{-1/[F:\QQ]} \big\},
\]
where the implicit constant is absolute.
\end{romanenum}
\end{prop}
\begin{proof}
For part (i), see \cite{Stark}*{Lemma 3}.  Part (ii) is due to Heilbronn, see \cite{Stark}*{Theorem 3} for a generalized version.  The estimate in part (iii) can be found in the proof of \cite{Stark}*{Theorem $1'$}.
\end{proof}

\begin{defn}
Let $L\neq \QQ$ be a number field.  If the simple real zero of $\zeta_L(s)$ described in Proposition~\ref{P:exceptional zeros}(i) exists, then we call it the \defi{exceptional zero} of $L$ and denote it by $\beta_L$.   Note that $\zeta_\QQ(s)$ has no real zeros in the interval $0\leq \sigma \leq 1$.
\end{defn}

\subsection{Unconditional versions}
\begin{prop} \label{P:unconditional CDT} 
Let $L/k$ be a Galois extension of number fields with Galois group $G$.   Let $C$ be a subset of $G$ that is stable under conjugacy and let $\widetilde{C}$ be the set of conjugacy classes of $G$ which are subsets of $C$.   
\begin{romanenum}
\item  There is an absolute constant $c_1>0$ such that if $\log x \geq 10 [L:\QQ] (\log d_L)^2$, then
\[
\Bigl| \pi_C(x,L/k) - \frac{|C|}{|G|} \Li x \Bigr| \leq \frac{|C|}{|G|} \Li(x^{\beta_L}) + O\Bigl( |\widetilde{C}| x \exp( -c_1 [L:\QQ]^{-1/2}(\log x)^{1/2})\Bigr),
\]
where the $\frac{|C|}{|G|} \Li(x^{\beta_L})$ term is present only when the exceptional zero $\beta_L$ exists.
\item There is an absolute constant $c_2>0$ such that if $\log x \geq c_2 (\log d_L)(\log\log d_L)(\log\log\log 6d_L)$, then
\[
\pi_C(x,L/k) \ll \frac{|C|}{|G|} \frac{x}{\log x}.
\]
\end{romanenum}
\end{prop}
\begin{proof}
Part (i) is a consequence of \cite{LO Chebotarev}*{Theorem 1.3}.  Part (ii) is stated as in \cite{SerreCheb}*{Th\'eor\`eme~3} and is a result of Lagarias, Montgomery, and Odlyzko.
\end{proof}

\begin{prop} \label{P:unconditional character sum}
Let $L/k$ be a Galois extension of number fields with Galois group $G$ and let $\chi$ be a character of $G$.  If $\log x \geq 10 [L:\QQ] (\log d_L)^2$, then
\begin{align*}
\sum_{\substack{\p\in\Sigma_k(x) \\ \p \text{ unramified in }L}} \chi(\Frob_\p) = (\chi,1) \Li x &+ O(\chi(1)\Li(x^{\beta_L}) )\\ 
& + O\Bigl(\chi(1) |G^\sharp |x\exp( -c_1  [L:\QQ]^{-1/2}(\log x)^{1/2})\Bigr)
\end{align*}
where the $\chi(1)\Li(x^{\beta_L})$ term is present only when the exceptional zero $\beta_L$ exists, and the constant $c_1>0$ and the implicit constants are absolute.
\end{prop}
\begin{proof}
It suffices to prove the proposition for an irreducible character $\chi$.  We first write the character sum in terms of the $\pi_C(x,L/k)$,
\[
\sum_{\substack{\p\in\Sigma_k(x) \\ \p \text{ unramified in }L}} \chi(\Frob_\p) = \sum_{C\in G^\sharp } \chi(C) \pi_C(x,L/k).
\]
Using $\sum_{C\in G^\sharp } \chi(C) {|C|}/{|G|} = (\chi,1)$ and $\max_{C\in G^\sharp }|\chi(C)| = \chi(1)$, the proposition follows directly from Proposition~\ref{P:unconditional CDT}(i).
\end{proof}

\section{Group theory for $\text{GSp}_{2g}$}
\subsection{Symplectic groups} \label{SS:symplectic group background}
Fix a field $k$, a finite dimensional vector space $V$ of dimension $2g$ over $k$, and a nondegenerate alternating bilinear form $\ang{ \; , \, }\colon V \times V \to k$.  We define $\GSp(V,\ang{\:, \,})$ to be the group of $A\in \Aut(V)$ such that for some $\mult(A) \in k^\times$, we have $\ang{Av,Aw} = \mult(A) \ang{v,w}$ for all $v,w \in V$.   Define $\Sp(V,\ang{\:, \,})$ to be the group of automorphisms of $V$ which preserve the pairing.   We call $\GSp(V,\ang{\:, \,})$ (resp.~$\Sp(V,\ang{\:, \,})$) the \defi{group of symplectic similitudes} (resp.~the \defi{symplectic group}).  The element $\mult(A)\in k^\times$ is called the \defi{multiplier} of $A$, and gives an exact sequence
\[
1\to \Sp(V,\ang{\:, \,}) \to \GSp(V,\ang{\:, \,}) \overset{\mult}{\to} k^\times \to 1.
\]
Up to isomorphism, $V$ has a unique non-degenerate alternating bilinear form; with this in mind, we may thus unambiguously use the notation $\GSp_{2g}(k)$ and $\Sp_{2g}(k)$.  Note that for $g=1$, $\GSp_{2}(k)=\GL_2(k)$ and $\Sp_{2}(k)=\SL_2(k)$.
For any $A\in\GSp_{2g}(k)$, we have the relation
\[
P(\mult(A)/T)/(\mult(A)/T)^g = P(T)/T^g,
\]
where $P(T)=\det(TI-A)\in k[T]$.

\subsection{Bounds on group orders and number of conjugacy classes}
\begin{lemma}[\cite{Gallagher}]  \label{L:Gallagher} If $G$ is a finite group and $N$ is a normal subgroup of $G$, then $|G^\sharp| \leq |(G/N)^\sharp||N^\sharp|.$
\qed
\end{lemma}

\begin{lemma}  \label{L:misc bounds for GSp}
Fix a prime power $q$.
\begin{romanenum}
\item $|\Sp_{2g}(\FF_q)| =q^{g^2}\prod_{i=1}^{g}(q^{2i}-1)$.
\item $|\GSp_{2g}(\FF_q)| =(q-1)q^{g^2}\prod_{i=1}^{g}(q^{2i}-1) \leq q^{2g^2 + g +1 }.$  
\item  There is a constant $\kappa_g$, depending only on $g$, such that
\[
|\Sp_{2g}(\FF_q)^\sharp| \leq \kappa_g q^g \text{\quad and \quad}
|\GSp_{2g}(\FF_q)^\sharp| \leq \kappa_g q^{g+1}.
\]
\end{romanenum}
\end{lemma}
\begin{proof}
Part (i) can be found in \cite{Artin}*{Chapter III \S6}, with part (ii) following immediately since $|\GSp_{2g}(\FF_q)| = (q-1)|\Sp_{2g}(\FF_q)|$.  We now consider part (iii).  By Lemma~\ref{L:Gallagher}, it suffices to prove the bound for $|\Sp_{2g}(\FF_q)^\sharp|$; this follows from \cite{LP}.
\end{proof}

\begin{lemma}  \label{L:misc bounds for GL}
Let $q$ be a prime power.
\begin{romanenum}
\item $|\GL_2(\FF_q)| = q(q-1)^2(q+1).$
\item $|\GL_2(\FF_q)^\sharp| = q^2-1.$
\item For $t\in\FF_q^\times$, $|\{A\in\GL_2(\FF_q): \tr(A)=t\}| =  q(q^2-q-1).$
\item Assume $q$ is odd.  For all $t\in\FF_q$ and $d\in \FF_q^\times$,
\[|\{A \in \GL_2(\FF_q) : \det(A)=d, \tr(A)=t \}| = q\left(q+ \left(\tfrac{t^2-4d}{q}\right)\right),\]
where $\left(\tfrac{\cdot}{q}\right)$ is the Legendre symbol.
\end{romanenum}
\end{lemma}
\begin{proof}
One has an explicit description of the conjugacy classes of $\GL_2(\FF_q)$, see for example \cite{LangAlgebra}*{XVIII Table 12.4}.  The lemma is then a direct computation.
\end{proof}



\begin{bibdiv}
\begin{biblist}

\bib{Achter}{article}{
  author = {Achter, Jeff},
  title = {Split reductions of simple abelian varieties},
  note={With an appendix by Emmanuel Kowalski},
  note={arXiv:0806.4421v1 [math.NT]}
  year = {2008}
}

\bib{Artin}{book}{
   author={Artin, E.},
   title={Geometric algebra},
   series={Wiley Classics Library},
   note={Reprint of the 1957 original;
   A Wiley-Interscience Publication},
   publisher={John Wiley \& Sons Inc.},
   place={New York},
   date={1988},
   pages={x+214},
}

\bib{Bombieri}{book}{
   author={Bombieri, Enrico},
   title={Le grand crible dans la th\'eorie analytique des nombres},
   note={Avec une sommaire en anglais;
   Ast\'erisque, No. 18},
   publisher={Soci\'et\'e Math\'ematique de France},
   place={Paris},
   date={1974},
   pages={i+87},
}

\bib{Bombieri}{article}{
   author={Bombieri, Enrico},
   title={On exponential sums in finite fields. II},
   journal={Invent. Math.},
   volume={47},
   date={1978},
   number={1},
   pages={29--39},
}

\bib{Chavdarov}{article}{
   author={Chavdarov, Nick},
   title={The generic irreducibility of the numerator of the zeta function
   in a family of curves with large monodromy},
   journal={Duke Math. J.},
   volume={87},
   date={1997},
   number={1},
   pages={151--180},
   issn={0012-7094},
}

\bib{Cojocaru-Koblitz}{article}{
   author={Cojocaru, Alina Carmen},
   title={Reductions of an elliptic curve with almost prime orders},
   journal={Acta Arith.},
   volume={119},
   date={2005},
   number={3},
   pages={265--289},
}
    
\bib{CojocaruDavid-LT2}{article}{
  author={Cojocaru, Alina Carmen},  
  author={David, Chantal},
  title={Frobenius fields for elliptic curves},
  journal={Amer. J. Math. (to appear)},
}
   
\bib{Elkies87}{article}{
   author={Elkies, Noam D.},
   title={The existence of infinitely many supersingular primes for every
   elliptic curve over ${\bf Q}$},
   journal={Invent. Math.},
   volume={89},
   date={1987},
   number={3},
   pages={561--567},
}

\bib{Elkies92}{article}{
   author={Elkies, Noam D.},
   title={Distribution of supersingular primes},
   note={Journ\'ees Arithm\'etiques, 1989 (Luminy, 1989)},
   journal={Ast\'erisque},
   number={198-200},
   date={1991},
   pages={127--132 (1992)},
}


\bib{Gallagher}{article}{
   author={Gallagher, Patrick X.},
   title={The number of conjugacy classes in a finite group},
   journal={Math. Z.},
   volume={118},
   date={1970},
   pages={175--179},
   issn={0025-5874},
}

\bib{Hall}{article}{
  author = {Hall, Chris},
  title = {An open image theorem for a general class of abelian varieties},
  note={With an appendix by Emmanuel Kowalski},
  note={arXiv:0803.1682v1 [math.NT]}
  year = {2008}
}

\bib{IwaniecKowalski}{book}{
   author={Iwaniec, Henryk},
   author={Kowalski, Emmanuel},
   title={Analytic number theory},
   series={American Mathematical Society Colloquium Publications},
   volume={53},
   publisher={American Mathematical Society},
   place={Providence, RI},
   date={2004},
   pages={xii+615},
}

\bib{Kac}{book}{
   author={Kac, Mark},
   title={Statistical independence in probability, analysis and number
   theory. },
   series={The Carus Mathematical Monographs, No. 12},
   publisher={Published by the Mathematical Association of America.
   Distributed by John Wiley and Sons, Inc., New York},
   date={1959},
}

\bib{Koblitz}{article}{
   author={Koblitz, Neal},
   title={Primality of the number of points on an elliptic curve over a
   finite field},
   journal={Pacific J. Math.},
   volume={131},
   date={1988},
   number={1},
   pages={157--165},
}
  
\bib{Kowalski-Monodromy}{article}{
   author={Kowalski, E.},
   title={The large sieve, monodromy and zeta functions of curves},
   journal={J. Reine Angew. Math.},
   volume={601},
   date={2006},
   pages={29--69},
}

\bib{Kowalski-LS}{book}{
  author={Kowalski, E.},
  title={The large sieve and its applications: arithmetic geometry, random walks, discrete groups},
  publisher={Cambridge University Press},
  date={2008}
}


\bib{LOChebotarev}{article}{
   author={Lagarias, J. C.},
   author={Odlyzko, A. M.},
   title={Effective versions of the Chebotarev density theorem},
   conference={
      title={Algebraic number fields: $L$-functions and Galois properties
      (Proc. Sympos., Univ. Durham, Durham, 1975)},
   },
   book={
      publisher={Academic Press},
      place={London},
   },
   date={1977},
   pages={409--464},
}

\bib{Lang-ANT}{book}{
   author={Lang, Serge},
   title={Algebraic number theory},
   series={Graduate Texts in Mathematics},
   volume={110},
   edition={2},
   publisher={Springer-Verlag},
   place={New York},
   date={1994},
   pages={xiv+357},
}

\bib{LangAlgebra}{book}{
  author={Lang, Serge},
  title={Algebra},
  series={Graduate Texts in Mathematics},
  volume={211},
  edition={3},
  publisher={Springer-Verlag},
  place={New York},
  date={2002},
  pages={xvi+914},
}

\bib{Lang-Trotter}{book}{
   author={Lang, Serge},
   author={Trotter, Hale},
   title={Frobenius distributions in ${\rm GL}\sb{2}$-extensions},
   note={Distribution of Frobenius automorphisms in ${\rm
   GL}\sb{2}$-extensions of the rational numbers;
   Lecture Notes in Mathematics, Vol. 504},
   publisher={Springer-Verlag},
   place={Berlin},
   date={1976},
   pages={iii+274},
}

\bib{LangWeil}{article}{
   author={Lang, Serge},
   author={Weil, Andr{\'e}},
   title={Number of points of varieties in finite fields},
   journal={Amer. J. Math.},
   volume={76},
   date={1954},
   pages={819--827},
}

\bib{LP}{article}{
   author={Liebeck, Martin W.},
   author={Pyber, L{\'a}szl{\'o}},
   title={Upper bounds for the number of conjugacy classes of a finite
   group},
   journal={J. Algebra},
   volume={198},
   date={1997},
   number={2},
   pages={538--562},
}

\bib{MilneWaterhouse}{article}{
   author={Milne, J. S.},
   author={Waterhouse, W. C.},
   title={Abelian varieties over finite fields},
   conference={
      title={1969 Number Theory Institute (Proc. Sympos. Pure Math., Vol.
      XX, State Univ. New York, Stony Brook, N.Y., 1969)},
   },
   book={
      publisher={Amer. Math. Soc.},
      place={Providence, R.I.},
   },
   date={1971},
   pages={53--64},
}

\bib{Montgomery}{article}{
   author={Montgomery, Hugh L.},
   title={The analytic principle of the large sieve},
   journal={Bull. Amer. Math. Soc.},
   volume={84},
   date={1978},
   number={4},
   pages={547--567},
}

\bib{Murty-ANT}{book}{
   author={Murty, M. Ram},
   title={Problems in analytic number theory},
   series={Graduate Texts in Mathematics},
   volume={206},
   note={Readings in Mathematics},
   publisher={Springer-Verlag},
   place={New York},
   date={2001},
   pages={xvi+452},
}

\bib{MMS}{article}{
   author={Murty, M. Ram},
   author={Murty, V. Kumar},
   author={Saradha, N.},
   title={Modular forms and the Chebotarev density theorem},
   journal={Amer. J. Math.},
   volume={110},
   date={1988},
   number={2},
   pages={253--281},
}

\bib{Ribet-modular}{article}{
   author={Ribet, Kenneth A.},
   title={On $l$-adic representations attached to modular forms},
   journal={Invent. Math.},
   volume={28},
   date={1975},
   pages={245--275},
}

\bib{Serre-Inv72}{article}{
   author={Serre, Jean-Pierre},
   title={Propri\'et\'es galoisiennes des points d'ordre fini des courbes
   elliptiques},
   journal={Invent. Math.},
   volume={15},
   date={1972},
   number={4},
   pages={259--331},
}

\bib{Serre-rep}{book}{
   author={Serre, Jean-Pierre},
   title={Linear representations of finite groups},
   note={Translated from the second French edition by Leonard L. Scott;
   Graduate Texts in Mathematics, Vol. 42},
   publisher={Springer-Verlag},
   place={New York},
   date={1977},
   pages={x+170},
}

\bib{SerreCheb}{article}{
   author={Serre, Jean-Pierre},
   title={Quelques applications du th\'eor\`eme de densit\'e de Chebotarev},
   journal={Inst. Hautes \'Etudes Sci. Publ. Math.},
   number={54},
   date={1981},
   pages={323--401},
}

\bib{SerreTopics}{book}{
   author={Serre, Jean-Pierre},
   title={Topics in Galois theory},
   series={Research Notes in Mathematics},
   volume={1},
   note={Lecture notes prepared by Henri Darmon;
   With a foreword by Darmon and the author},
   publisher={Jones and Bartlett Publishers},
   place={Boston, MA},
   date={1992},
   pages={xvi+117},
}

\bib{SerreMotivicGalois}{article}{
   author={Serre, Jean-Pierre},
   title={Propri\'et\'es conjecturales des groupes de Galois motiviques et
   des repr\'esentations $l$-adiques},
   conference={
      title={Motives},
      address={Seattle, WA},
      date={1991},
   },
   book={
      series={Proc. Sympos. Pure Math.},
      volume={55},
      publisher={Amer. Math. Soc.},
      place={Providence, RI},
   },
   date={1994},
   pages={377--400},
}

\bib{SerreMordellWeil}{book}{
  author={Serre, Jean-Pierre},
  title={Lectures on the Mordell-Weil theorem},
  series={Aspects of Mathematics},
  edition={3},
  note={Translated from the French and edited by Martin Brown from notes by Michel Waldschmidt; With a foreword by Brown and Serre},
  publisher={Friedr. Vieweg \& Sohn},
  place={Braunschweig},
  date={1997},
  pages={x+218},
}

\bib{Serre-VolumeIV}{book}{
   author={Serre, Jean-Pierre},
   title={\OE uvres. Collected papers. IV},
   note={1985--1998},
   publisher={Springer-Verlag},
   place={Berlin},
   date={2000},
   pages={viii+657},
}

\bib{Stark}{article}{
   author={Stark, H. M.},
   title={Some effective cases of the Brauer-Siegel theorem},
   journal={Invent. Math.},
   volume={23},
   date={1974},
   pages={135--152},
}

\bib{Zywina-Koblitz}{article}{
  author={Zywina, David},
  title={A refinement of Koblitz's conjecture},
  journal={preprint},
  date={2008},
}

\bib{Zywina-LT2}{article}{
  author={Zywina, David},
  title={The Lang-Trotter conjecture and mixed representations},
  journal={preprint},
  date={2008},
}

\end{biblist}
\end{bibdiv}
\end{document}